\documentclass[a4paper, 11pt]{article}
\usepackage{latexsym,amsfonts, amsmath, amssymb, amsthm}
\usepackage[utf8]{inputenc}
\usepackage[final=true,protrusion=true,expansion=true]{microtype}
\usepackage[noadjust]{cite}
\usepackage{enumitem}
\usepackage{color}
\usepackage{mathtools}
\usepackage{bm}

\usepackage{booktabs}
\usepackage{multirow}

\usepackage{subcaption}
\usepackage[font=scriptsize,labelfont=bf]{caption}
\usepackage[affil-it]{authblk}

\usepackage{geometry}
\usepackage{hyperref}
\hypersetup{
	final,
    colorlinks=true,
    linkcolor=blue,
    filecolor=magenta,
    urlcolor=cyan,
}
\usepackage[nameinlink,noabbrev,sort&compress]{cleveref}
\crefname{subsection}{subsection}{subsections}
\crefname{assumption}{assumption}{assumptions}

\usepackage{stmaryrd}

\theoremstyle{plain}
\newtheorem{theorem}{Theorem}[section]

\newtheorem{lemma}[theorem]{Lemma}
\newtheorem{corollary}[theorem]{Corollary}

\theoremstyle{plain}

\newtheorem{remark}[theorem]{Remark}
\newtheorem{assumption}[theorem]{Assumption}
\newcommand{\dummy}{\mathord{\color{black!33}\bullet}}

\providecommand{\Supp}{\operatorname{supp}}                            % support
\providecommand{\supp}{\Supp}
\providecommand{\argmin}{\operatorname*{arg\,min}}  % argument yielding inf
\providecommand{\Ce}{{\cal E}}

\providecommand{\CC}{{\cal C}}

\providecommand{\CE}{{\cal E}}

\providecommand{\CO}{{\cal O}}
\providecommand{\CP}{{\cal P}}

\providecommand{\CV}{{\cal V}}
\providecommand{\CW}{{\cal W}}

\providecommand{\bbE}{\mathbb{E}}

\providecommand{\bbN}{\mathbb{N}}

\providecommand{\bbP}{\mathbb{P}}

\providecommand{\bbR}{\mathbb{R}}

\providecommand*{\N}[1]{\left\|{#1}\right\|} % Double bar norm
\newcommand*{\SN}[1]{\left|{#1}\right|}      % Single bar norm
\newcommand{\MagPar}[1]{\ensuremath{\textcolor{black}{\left( #1 \right)}}} % relevant: contains macros for norms and partial derivatives

\usepackage{bbm} % Black board math
\usepackage{ifdraft}
\usepackage{arydshln}
\usepackage{subcaption}

\usepackage{mdframed}
\usepackage{xcolor}
\usepackage{soul}

\usepackage{tikz}
\usepackage{scalerel}
\usepackage{tikz-cd}
\usetikzlibrary{arrows,intersections,calc,patterns,angles,quotes}
\usetikzlibrary{shapes.geometric}
\usetikzlibrary{calc}
\usetikzlibrary{shapes,arrows}

\usepackage[textsize=footnotesize]{todonotes}

\newcommand{\tn}{\textnormal}

\newcommand{\overbar}[1]{\mkern 1.5mu\overline{\mkern-1.5mu#1\mkern-1.5mu}\mkern 1.5mu}

\newcommand{\overbarsuperscriptt}[1]{\mkern 1.5mu\overline{\mkern-1.5muX_{#1}\mkern-1.5mu}\mkern1.5mu}
\newcommand{\overbarsuperscriptit}[1]{\mkern -1.5mu{\mkern 1.5mu{\mkern-1.5mu\overline{X}^i_{#1}\mkern-1.5mu}\mkern1.5mu}}
\newcommand{\overtildesuperscriptit}[1]{\mkern -1.5mu{\mkern 1.5mu{\mkern-1.5mu\widetilde{X}^i_{#1}\mkern-1.5mu}\mkern1.5mu}}
\newcommand{\hatsuperscriptit}[1]{\mkern -1.5mu{\mkern 1.5mu{\mkern-1.5mu\widehat{X}^i_{#1}\mkern-1.5mu}\mkern1.5mu}}

\renewcommand{\underbar}[1]{\mkern 1.5mu\underline{\mkern-1.5mu#1\mkern-1.5mu}\mkern 1.5mu}

\DeclarePairedDelimiter\abs{\lvert}{\rvert}%
\DeclarePairedDelimiter\absbig{\big\lvert}{\big\rvert}%

\renewcommand{\abs}[1]{\left|{#1}\right|}
\providecommand*{\absbig}[1]{\big|{#1}\big|} % abs bar norm always big
\DeclareMathOperator*{\Law}{Law}

\newcommand{\setnumber}[1]{\llbracket #1 \rrbracket}

\def\R{\mathbb{R}}
\def\Rd{\mathbb{R}^d}
\def\Rn{\mathbb{R}^n}
\def\B{B} % may be later replaced by \def\B{B^{\infty,}}

\def \d{\textup{d}}

\newcommand{\Xbar}[1]{\overbarsuperscriptt #1} % for t-mf-CBO
\newcommand{\Xibar}[1]{\overbarsuperscriptit #1} % for copies of t-mf-CBO
\newcommand{\Xihat}[1]{\hatsuperscriptit {#1}} % for EM-CBO
\newcommand{\Xitilde}[1]{\overtildesuperscriptit {#1}} % for t-EM-CBO
\newcommand{\W}[1]{W^i_{#1}}% for EM-CBO
\newcommand{\dt}[0]{\Delta t}
\newcommand{\disp}[1]{\gamma_{\dt}(#1)} % dp stands for discrete point
\newcommand{\Xtilde}[1]{\widetilde{X}^{#1}}

\newcommand{\cprel}[0]{c_{\tn{B}}} % constant of prel
\newcommand{\ccor}[0]{C_{\tn{B}}} % constant of coroll of prel
\newcommand{\cmom}[0]{\widetilde{C}} % constant in bound on moment of t-EM-CBO
\newcommand{\cbm}[0]{c_{d,q}} % constant in bound on moment of bm
\newcommand{\cbt}[0]{c_{d,2}}
\usepackage{accents}
\newcommand{\dbtilde}[1]{\accentset{\approx}{#1}}
\newcommand{\cgap}[0]{\dbtilde{C}} % constant in gap t-EM-CBO
\newcommand{\cbgd} [0] {c_{\tn{BDG,q}}}
\newcommand{\cna}[0]{C_{\tn{NA}}} 
\newcommand{\cmfa}[0]{C_{\tn{MFA}}}
\newcommand{\cstab}[0]{c_\tn{S}}

\newcommand{\cbgdt}[0]{c_{\tn{BDG,2}}}
\newcommand{\cwm}[0]{c_\tn{C}}
\newcommand{\cle}[0]{c_\tn{LE}}
\newcommand{\cbmmf}[0]{c_\tn{Bmf}}

\newcommand{\globmin}{x^*}
\newcommand{\minobj}{\underbar \CE}
\newcommand{\maxobj}{\overbar \CE}

\newcommand{\drift}{b}
\newcommand{\diff}{n}
\newcommand{\stoc}[0]{D}
\newcommand{\stocis}[0]{D^{\text{iso}}}
\newcommand{\stocan}[0]{D^{\text{aniso}}}

\newcommand{\Le}{L_{\CE}}
\newcommand{\cu}{c_u}
\newcommand{\cl}{c_l}
\newcommand{\Cll}{\tilde{c}_l}

\newcommand{\indivmeasure}[0]{\varrho} %\widetilde\rho already used

\newcommand{\empmeasure}[1]{\widehat\rho_{#1}^N}

\newcommand{\monopmeasure}[1]{\overline\rho_{#1}^N}

\newcommand{\tilempmeasure}[1]{\widetilde\rho_{#1}^N}

\newcommand{\omegaa}[0]{\omega_{\alpha}}

\newcommand{\conspoint}[1]{x_{\alpha}({#1})}

\makeatother

\title{\usefont{OT1}{bch}{b}{n}
	\LARGE Strong Global Convergence of the\\Consensus-Based Optimization Algorithm\\
}
\author[Sabrina Bonandin, Konstantin Riedl and Sara Veneruso]{}
\date{}

\begin{document}
\maketitle

% Authors, emails and affiliations
\vspace{-1.8cm}
\centerline{Sabrina Bonandin$^{1}$, Konstantin Riedl$^{2}$ and Sara Veneruso$^{1,3}$}

\bigskip

{\footnotesize
% Enter the full affiliation and country name:
% Do not insert commas or periods at the end of lines.
 \centerline{$^1$Institute for Geometry and Practical Mathematics, RWTH Aachen University,}
 \centerline{Templergraben 55, 52062 Aachen, Germany}
} 

\medskip

{\footnotesize
% Enter the full affiliation and country name:
% Do not insert commas or periods at the end of lines.
 \centerline{$^2$Mathematical Institute, University of Oxford,}
 \centerline{Radcliffe Observatory, Andrew Wiles Building, Woodstock Rd, Oxford OX2 6GG, United Kingdom}
} 

\medskip

{\footnotesize
 % Enter the full affiliation and country name:
  \centerline{$^3$Department of Mathematics and Computer Science, University of Ferrara,}
   \centerline{Via Machiavelli 30, 44121 Ferrara, Italy}
}

\medskip
{\footnotesize

 \centerline{Email addresses: \href{mailto:bonandin@eddy.rwth-aachen.de}{\texttt{bonandin@eddy.rwth-aachen.de}}, \href{mailto:konstantin.riedl@maths.ox.ac.uk}{\texttt{konstantin.riedl@maths.ox.ac.uk}},}
 \centerline{\href{mailto:sara.veneruso@unife.it}{\texttt{sara.veneruso@unife.it}}
}

%%%%%%%%%%%%%%%%%%%%%%%%%%%%%%%%%%%%%%%%%%%%%%%%%%
%%%%%%%%%% Abstract %%%%%%%%%%%%%%%%%%%%%%%%%%%%%%
%%%%%%%%%%%%%%%%%%%%%%%%%%%%%%%%%%%%%%%%%%%%%%%%%%

\begin{abstract}
\noindent
Consensus-based optimization (CBO) is a multi-agent metaheuristic derivative-free optimization algorithm that has proven to be capable of globally minimizing nonconvex nonsmooth functions across a diverse range of applications while being amenable to theoretical analysis. 
The method leverages an interplay between exploration of the energy landscape of the objective function through a system of interacting particles subject to stochasticity and exploitation of the particles' positions through the computation of a global consensus about the location of the minimizer based on the Laplace principle.
In this paper, we prove strong mean square convergence of the practical numerical time-discrete CBO algorithm to the global minimizer for a rich class of objective functions.
For CBO with both isotropic and anisotropic diffusion, our convergence result features conditions on the choice of the hyperparameters as well as explicit rates of convergence in the time discretization step size $\dt$ and the number of particles~$N$.
By interpreting the time-discrete algorithm at the continuous-time level through a system of stochastic differential equations (SDEs), our proof strategy combines traditional finite-time convergence theory for numerical methods applied to SDEs with careful considerations due to the fact that the CBO coefficients do not satisfy a global Lipschitz condition.
To accommodate the latter, we adopt a recently proposed generalization of Sznitman’s classical argument, which allows to discard an event of small probability, controllable through fine moment estimates for the particle systems.
\end{abstract}

%%%%%%%%%% Keywords %%%%%%%%%%%%%%%%%%%%%%%%%%%%%%
{\noindent\small{\textbf{Keywords:} global optimization, derivative-free optimization, metaheuristics, consensus-based optimization, strong convergence, algorithm, mean-field limit, Euler-Maruyama}}\\

{\noindent\small{\textbf{AMS subject classifications:} 65K10, 90C26, 90C56, 60H10, 65C30, 65C20}} 

\maketitle

%%%%%%%%%%%%%%%%%%%%%%%%%%%%%%%%%%%%%%%%%%%%%%%%%%
%%%%%%%%%% Section %%%%%%%%%%%%%%%%%%%%%%%%%%%%%%%
%%%%%%%%%%%%%%%%%%%%%%%%%%%%%%%%%%%%%%%%%%%%%%%%%%
\section{Introduction}
\label{sec:intro}

Many challenges encountered across various quantitative fields, including operations research, control theory, engineering, economics, and machine learning, involve tackling global unconstrained optimization problems of the form
\begin{equation*}
    \label{eq:min_prob}
    \globmin = \argmin_{x\in \Rd} \CE(x),
\end{equation*}
where $\CE: \Rd \to \R$, $d \ge 1$, denotes a potentially nonconvex nonsmooth and high-dimensional objective function. 
In scenarios where the lack of regularity or the availability of the objective function $\CE$ as a black box limits the use of traditional gradient-based optimization methods, derivative-free (zero-order) optimization algorithms~\cite{conn09intro}, i.e., methods that rely solely on the evaluation of the objective function $\CE$, have been and are a popular alternative.
Among these, so-called metaheuristics have proven to be surprisingly effective due to their ease of conceptualization, implementation and parallelizability \cite{back1997handbook,blum2003metaheuristics}.
Such methods balance a deterministic exploitation mechanism, aimed at locating regions in the search space that contain high quality approximations to the solution,
and a stochastic exploration feature that encourages the search of the space of admissible solutions, thus enabling the escape from local optima.
Some notable examples include, amongst others, random search \cite{rastrigin1963convergence}, simulated annealing \cite{aarts1989simulated}, genetic algorithms \cite{reeves2010genetic}, and particle swarm optimization \cite{kennedy1995particle}.
Despite their considerable empirical success, their complexity (a consequence of the interplay between stochasticity and potentially intricate decision rules) often results in a lack of rigorous convergence guarantees.
Nonetheless, advancements in this field have been achieved over the last few years through a relatively recent metaheuristic known as consensus-based optimization (CBO) \cite{pinnau2017consensus}.
In the spirit of the two competing factors, exploitation and exploration, CBO algorithms utilize a finite number of $N \in \bbN^+$ particles (also referred to as agents) that explore the landscape of the objective function $\CE$.
Each particle is driven by a deterministic drift and a stochastic diffusion component,
and, as time passes, they concentrate around a so-called consensus point, which serves as a tentative approximation of the global minimizer $\globmin$.
A global convergence analysis of CBO can be conducted either in the finite-particle regime (also referred to as the microscopic level), as done, for instance, in \cite{bellavia2025discrete}, or in the limit as the number of agents~$N$ approaches infinity, the so-called mean-field regime, where a statistical description of the dynamics of the average agent behavior is considered, see, e.g., \cite{pinnau2017consensus,carrillo2018analytical,fornasier2024consensus,riedl2024perspective}, to name a few. 
The gap between the two regimes has been extensively studied recently, leading to the establishment of what are known as quantitative mean-field results, see \cite{fornasier2020consensus_hypersurface_wellposedness,fornasier2024consensus,gerber2023mean,gerber2025uniform,kalise2022consensus,bonandin2026exploiting} to give a few examples.

Let us now provide a formal numerical description of the CBO algorithm.
For a fixed number of iterations $K$ and a time step size $\dt>0$,
we denote by $\Xihat{k \dt} \in \Rd$ the position of agent $i \in \setnumber{N}:=\{1,\ldots,N\}$ at time $k \dt$ for $k = 0,\ldots,K$.%
\footnote{To align our notation with the fact that the CBO dynamics is typically formulated in continuous time, we index the positions of the particles by $k\dt$ (rather than by $k$). Thus, each discrete iterate of the numerical scheme corresponds to the physical time $k\dt$.} 
Starting from randomly initialized initial positions~$\Xihat{0}$,
the algorithm's iterative update rule is given for all  $k = 0,\ldots,K-1$ by
\begin{equation}
	\label{eq:EM-CBO}
		\Xihat{(k+1) \dt} = \Xihat{k \dt} 
        -\lambda \left( \Xihat{k \dt} - \conspoint{\empmeasure{k \dt}} \right) \dt
        + \sigma \stoc \left( \Xihat{k \dt} - \conspoint{\empmeasure{k \dt}} \right) 
        \Delta \W{k\dt},
\end{equation}
where $\alpha,\lambda,\sigma >0$ denote user-specified hyperparameters.
$\Delta \W{k\dt} \coloneqq \W{(k+1) \dt}-\W{k \dt}$ are the increments of the Brownian motion $W^i$, i.e., $\Delta \W{0}, \ldots, \Delta \W{(K-1)\dt}$ are independent with $\Delta \W{k\dt}$ being a normally distributed $d$-dimensional random vector with mean zero and covariance matrix $\dt \cdot I_d$, where $I_d \in \R^{d \times d}$ is the identity matrix. 
The particles interact at time $k\dt$ through their empirical measure~$\empmeasure{k\dt}$, and $x_{\alpha}$ denotes the consensus point, which is given for any probability distribution $\varrho$ on $\bbR^d$ by
\begin{equation*}
    \label{eq:conspoint}
    \conspoint{\indivmeasure} \coloneqq \int_{\Rd} x \frac{\omegaa(x)}{\N{\omegaa}_{L^1(\indivmeasure)}} \,\d\indivmeasure(x)
    \quad \text{ with }\quad
    \omegaa(x) \coloneqq \exp(-\alpha\CE(x)).
\end{equation*}
It can be demonstrated that it provides an accurate approximation of $\globmin$ through the well-known Laplace principle \cite{dembo2009large,miller2006applied}, according to which it holds 
\begin{equation*}
    \lim_{\alpha \to \infty} \left( -\frac{1}{\alpha} \log \left( \int_{\Rd} \omegaa(x) \,\d\indivmeasure(x) \right) \right) = \inf_{x \in \tn{supp}(\indivmeasure)} \CE(x).
\end{equation*}
In \eqref{eq:EM-CBO}, the term governed by $\lambda$ represents the deterministic drift term, which is designed to pull each agent towards $x_{\alpha}$.
On the other hand, the term multiplied by $\Delta \W{k\dt}$ is the stochastic diffusion term that facilitates exploration of the landscape of the objective function, and is regulated by the parameter $\sigma$ and the scaling transformation $\stoc$.
$\stoc:\Rd \to \mathbb{R}^{d\times d}$ implements a continuous function that was initially introduced in the pioneering work on CBO \cite{pinnau2017consensus} in the isotropic form $\stoc(\dummy) = \stocis(\dummy)\coloneqq |\dummy|\,I_d$, with $\SN{\dummy}$ denoting the Euclidean norm in $\Rd$. 
Subsequently, in order to enhance the feasibility and competitiveness of CBO in large-scale and high-dimensional applications, beginning with \cite{carrillo2021consensus}, it was replaced by its component-wise anisotropic counterpart $\stoc(\dummy) = \stocan(\dummy)\coloneqq \mathrm{diag}(\dummy)$, with $\tn{diag}: \Rd \to \R^{d\times d}$ being the operator mapping a vector of $\Rd$ onto a diagonal matrix with the vector as its diagonal. 

Examples demonstrating the applicability of CBO to such high-dimensional problems include compressed sensing \cite{riedl2022leveraging}, phase retrieval~\cite{fornasier2020consensus_sphere_convergence}, robust subspace detection \cite{fornasier2020consensus_sphere_convergence}, sparse representation learning \cite{trillos2024CB2O}, neural network training for image classification \cite{carrillo2019consensus,fornasier2021convergence,riedl2022leveraging,borghi2023consensus}, clustered federated learning \cite{carrillo2024fedcbo,trillos2024attack}, adversarial training~\cite{trillos2024attack,roith2025consensus}, and the optimization of qubit configurations~\cite{de2025consensus}.
Code for CBO is available at \cite{bailo2024cbx}.

\subsection{Contributions and main result}
\label{sec:contributions}

Motivated by the aforementioned wide range of applications of the CBO method,
we present in this paper a quantitative convergence result for the implementable algorithm \eqref{eq:EM-CBO} for CBO with both isotropic and
anisotropic diffusion.
To present our results for both cases concisely, we introduce the constant
\begin{equation}
\label{def:kappa}
    \kappa(\stoc) \coloneqq
    \begin{cases}
        d \quad &\tn{if $\stoc = \stocis$},\\
        1 \quad &\tn{if $\stoc = \stocan$}.
    \end{cases}
\end{equation}
More precisely, we prove strong global mean square convergence (see, e.g., \cite[Theorem 7.11]{graham2013stochastic} or \cite[Theorem 10.6.4]{kloeden1992numerical}) of the solution $\Xihat{k \dt}$ of the numerical scheme to the global minimizer $\globmin$ with an explicit rate.

\begin{theorem}
    \label{thm:gap_tEM_globmin} 
    Suppose that $\CE$ satisfies \Cref{ass:wp,ass:ICP}.
    Let $\rho_0 \in \CP_q(\Rd)$ for some $q\ge 4$ be such that $\globmin \in \tn{supp}(\rho_0)$ and initialize $\Xihat{0}\sim\rho_0$.
    Define $\CV(\rho_0) \coloneqq \frac{1}{2} \int_{\Rd} \SN{x-\globmin}^2 \d \rho_0(x)$. Fix a time step size $0 < \dt \le 1$ and a total number of iterations $K \in \bbN^+$, and let
    $\{\Xihat{k \dt}\}_{k\in\setnumber{K}}^{i\in\setnumber{N}}$ 
    %$((\Xihat{k \dt})_{k\in\setnumber{K}})_{i\in\setnumber{N}}$ 
    denote the iterates generated by the CBO algorithm \eqref{eq:EM-CBO}.
    Let $\vartheta \in (0,1)$, and choose hyperparameters $\lambda, \sigma > 0$ satisfying $2\lambda - \kappa(D) \sigma^2 > 0$, with $\kappa(\stoc)$ defined in \eqref{def:kappa}.  
    Then there exists $\alpha_0>0$ 
    such that for all $\alpha\geq\alpha_0$ it holds for any $i \in \setnumber{N}$ that
    \begin{equation}
    \label{eq:gap_tEM_globmin}
    \begin{split}
        \bbE \SN{\Xihat{K \dt} - \globmin}^2
        \leq \cna\Delta t + \cmfa N^{-\min \left\{1, \frac{q-2}{4} \right\}} + 4 \CV(\rho_0)\exp\left(-(1-\vartheta)\big(2\lambda - \kappa(D) \sigma^2\big)K\Delta t\right).
    \end{split}
    \end{equation}
    The constants $\cna$ and $\cmfa$, independent of both $N$ and $\dt$, but dependent on 
    $K \dt$, are defined in \Cref{thm:gap_tEM_tmfCBO}.
\end{theorem}

\begin{corollary}
    \label{cor:gap_tEM_globmin}
    Fix $\varepsilon_{\text{total}} > 0$.
    Under the assumptions of \Cref{thm:gap_tEM_globmin},
    there exist parameter choices $\alpha_0$, $N$, $\dt$, and $K$ such that for any $i \in \setnumber{N}$ it holds
    \begin{equation*}
        \bbE \SN{\Xihat{K \dt} - \globmin}^2
        \leq \varepsilon_{\text{total}}.
    \end{equation*}
    This can be achieved by choosing the parameters $\alpha_0$, $N$, $\dt$, and $K$ such that all of the following hold:
    \begin{enumerate}[label=(\roman*)]
        \item selecting $\alpha_0$ as specified in \cite[Theorem~3.7]{fornasier2024consensus} in the case of isotropic noise, and \cite[Theorem~3.6]{riedl2024perspective} in the case of anisotropic noise, with $\varepsilon$ as given in 
        \begin{equation*}
            \varepsilon = \CV(\rho_0) \left( \frac{\varepsilon_{\text{total}}}{12 \CV(\rho_0)}\right)^{\frac{1+\vartheta/2}{1-\vartheta}};
        \end{equation*}
        \item fixing a time horizon
        \begin{equation}
        \label{eq:defT}
            T\geq \frac{1}{(1-\vartheta)(2\lambda - \kappa(D) \sigma^2)} \log \left(\frac{12 \CV(\rho_0)}{\varepsilon_{\text{total}}}\right)+1;
        \end{equation}
        \item and choosing
        \begin{equation*}
            N \ge \left( \frac{3\cmfa}{\varepsilon_{\text{total}}} \right)^{\max \left\{1, \frac{4}{q-2} \right\}}, \quad \dt \le \frac{\varepsilon_{\text{total}}}{3\cna}, \quad \text{and} \quad K = \left\lceil \frac{T}{\dt}\right\rceil.
        \end{equation*}
    \end{enumerate}
\end{corollary}

%%%%%%%%%%%%%%%%%%%%%%%%%%
\Cref{thm:gap_tEM_globmin} provides an upper bound for the expected squared distance between the solution to the implementable CBO algorithm~\eqref{eq:EM-CBO} and the global minimizer $\globmin$ that holds for any $\alpha \ge \alpha_0$, an arbitrary particle $i \in \setnumber{N}$, and that depends explicitly on the time discretization step size $\dt$, the number of particles $N$, and the final iteration count $K$. 

A straightforward application of Jensen's inequality enables us to derive the same upper bound as in \eqref{eq:gap_tEM_globmin} for the empirical average of the particles' positions, i.e., a bound on $\bbE \SN{\frac{1}{N} \sum_{i=1}^N \Xihat{K \dt} - \globmin}^2$. 

In \Cref{cor:gap_tEM_globmin}, we explicitly outline how to practically select the aforementioned relevant hyperparameters in order to achieve a desired accuracy $\varepsilon_{\text{total}}$. In addition, whenever we refer to the time horizon $T$ in the remainder of the manuscript, we mean that it is chosen according to Equation \eqref{eq:defT}.
%\konst{SB: added comment about what we mean by $T$ in the rest of the manuscript.}

To the best of our knowledge, these two results are the first complete practical results demonstrating global convergence in the expected mean squared sense to the global minimizer~$\globmin$ for the implementable CBO algorithm~\eqref{eq:EM-CBO}.
In particular, our manuscript improves upon existing results found in the literature, which we outline as follows.
\begin{itemize}
\item The works \cite{ha2020convergenceHD} and \cite{kalise2022consensus} demonstrate that algorithm \eqref{eq:EM-CBO} achieves stochastic consensus (specifically, in our notation, that there exists the finite limit $\lim_{k \to \infty} \Xihat{k \dt}$ in a suitable probabilistic sense for any $i \in \setnumber{N}$) and that the convergence to it is exponentially fast under suitable assumptions about the system parameters. However, their results necessitate a choice of parameters that implies that the algorithm's convergence can only be established if there is a good initial estimate of the global minimizer, hence suggesting a local, not global, convergence to $\globmin$ (see \cite[Section~2.1]{fornasier2024consensus} for a more detailed discussion on the topic). Additionally, the proof technique they employ prevents them from obtaining an explicit rate in $\dt$.
Furthermore, \cite{ha2020convergenceHD} assumes that all particles share the same Brownian motion.
Our \Cref{thm:gap_tEM_globmin} addresses these limitations. 
\item \cite[Theorem~3.8]{fornasier2024consensus} in the case of isotropic noise, and \cite[Theorem~3.19]{riedl2024perspective} in the case of anisotropic noise, present quantitative global convergence results for \eqref{eq:EM-CBO} (with an explicit rate in $\dt$ and $N$), but in a probabilistic framework. More specifically, they demonstrate, in our notation and up to some constant factors, that 
\begin{equation*}
    \SN{\frac{1}{N} \sum_{i=1}^N \Xihat{K \dt} - \globmin}^2 \le \varepsilon_{\text{total}}
\end{equation*}
with probability larger than $1 - \left( \delta + \varepsilon_{\text{total}}^{-1} \left( \cna \Delta t + \cmfa N^{-1} + \varepsilon \right) \right)$, for a desired accuracy $\varepsilon_{\text{total}}$ and a parameter $\delta \in (0,1/2)$. 
$\cna$ is a positive constant depending linearly on $d$, 
exponentially on $T, \delta^{-1}$, and $\cmfa$ is a positive constant depending exponentially on $\alpha, \lambda, \sigma, T, \delta^{-1}$. 
Recalling that global convergence of CBO methods can be derived at either the algorithmic or microscopic level, the authors of \cite{fornasier2024consensus,riedl2024perspective} adopt the latter approach. They achieve the aforementioned result by combining a probabilistic quantitative mean-field result with a global convergence result in mean-field law through a triangular argument.
We remark that, as strong mean square convergence implies convergence in probability, our 
\Cref{thm:gap_tEM_globmin} generalizes the two aforementioned theorems. 
\item 
The work \cite{gerber2023mean} builds upon the probabilistic quantitative mean-field result of \cite{fornasier2024consensus} and introduces a non-probabilistic formulation. Their proof technique is based on a generalization of classical arguments to derive quantitative mean-field results, see McKean's and, later, Sznitman's arguments \cite{mckean1967propagation,sznitman1991propagation,chaintron2022propagation},
which cannot be directly applied due to the lack of global Lipschitz continuity in the coefficients of the CBO, and it involves discarding an event with a small, controllable probability.
They present their quantitative mean-field result for the continuous-time approximation of the implementable CBO's update rule. However, their estimate is insufficient to guarantee practical convergence of the algorithm \cite{gerber2025uniform}. Our \Cref{thm:gap_tEM_globmin} addresses this issue by providing a comprehensive convergence result that combines a novel quantitative mean-field estimate for the implementable CBO algorithm \eqref{eq:EM-CBO} with a global convergence result in mean-field law.
\end{itemize}

\subsection{Organization}
The rest of the paper is organized as follows. 
In \Cref{sec:main_result}, we discuss our main result presented in \Cref{thm:gap_tEM_globmin}. This includes an overview of the key assumptions required, a presentation of the proof strategy, and a detailed proof of the result.
Our main finding depends on several auxiliary results, which are collected and proven in \Cref{sec:auxiliary}. 
In the appendix, we provide various technical supplementary findings necessary for \Cref{sec:auxiliary}, alongside a summary of the constants that appear throughout this work.

\subsection{Notation}
We set $\setnumber{N}:=\{1,\ldots,N\}$, for any positive natural number $N \in \bbN^+$.
The Euclidean and $\ell^\infty$ norms of a vector $u \in \Rd$ are denoted by $\SN{u}$ and $\N{u}_{\infty}$, respectively, and the corresponding balls are written as $\B_{r}(u) := \{y \in \bbR^d: \SN{y-u} \leq r\}$ and $\B^{\infty}_{r}(u) := \{y \in \bbR^d: \N{y-u}_{\infty} \leq r\}$. For a matrix $A \in \R^{d\times d}$, the notation $\N{A}_F$ refers to the Frobenius norm. 

Given a random variable $X$, $\bbE(X)$ denotes its expectation.
This manuscript focuses mainly on time-dependent stochastic processes, and we indicate by $\Omega$ the sample space on which they are defined.  
The set of probability measures over a metric space $E$ is represented by $\CP(E)$, and the probability measures with finite moments up to order $1\leq r<\infty$ are collected in $\CP_r(E) \subset \CP(E)$.  
$\CP_{r,R}(E) \subset \CP_r(E)$ denotes the set of probability measures with $r$-th moment bounded by $R >0$.
When discussing a particular fixed distribution, we write~$\indivmeasure$.
$\CW_r$
is the standard \mbox{Wasserstein-$r$} distance (see, e.g., \cite{savare2008gradientflows}).

For the space of continuous functions~$f:X\rightarrow Y$ we write $\CC(X,Y)$, with $X \subset \Rn$ and $Y$ a suitable topological space. 
In the case of real-valued functions we omit $Y$. 

{\color{black} For a positive integer $n$, $n!!$ denotes the double factorial, defined as the product $n(n-2)(n-4)\cdots$, terminating at $1$ if $n$ is odd and at $2$ if $n$ is even.}

We use the symbol $\lesssim$ to denote inequalities that hold up to a constant multiplicative factor which is independent of problem-dependent constants.
Whenever we state that a constant depends on $\CE$, we mean that it relies on some of the constants present in \Cref{ass:wp,ass:ICP}.

%%%%%%%%%%%%%%%%%%%%%%%%%%%%%%%%%%%%%%%%%%%%%%%%%%
%%%%%%%%%% Section %%%%%%%%%%%%%%%%%%%%%%%%%%%%%%%
%%%%%%%%%%%%%%%%%%%%%%%%%%%%%%%%%%%%%%%%%%%%%%%%%%
\section{Discussion of the main result}
\label{sec:main_result}

In this section, we first discuss the assumptions on the objective function $\CE$ that will be used throughout the article.
Subsequently, we outline our proof strategy for establishing our main result, \Cref{thm:gap_tEM_globmin}.
This necessitates the introduction of several auxiliary continuous-time processes as well as an intermediate result in \Cref{thm:gap_tEM_tmfCBO}, whose proof is deferred to \Cref{sec:auxiliary}. We  conclude this section with the proof of \Cref{thm:gap_tEM_globmin}.

\subsection{Assumptions}
\label{sec:key_assumptions}

Let us first specify and discuss the assumptions on the objective function that form the basis for our theoretical analysis.
\begin{assumption}
    \label{ass:wp}
    The objective function $\CE \in \CC(\Rd)$
    \begin{enumerate}[label=\textbf{A\arabic*},labelsep=10pt,leftmargin=35pt, ref={A\arabic*}]
		\item\label[assumption]{ass:wp_lowerbound} 
        is such that there exists a unique $\globmin \in \Rd$ with $\CE(\globmin) = \inf_{x \in \Rd}\CE(x) \eqqcolon \minobj  >- \infty$,
        \item\label[assumption]{ass:wp_lipschitz} 
        satisfies for some constant $\Le> 0$ the condition
        \begin{align*}
        \SN{\CE(x) - \CE(y)} &\le \Le \,(1 + \SN{x} + \SN{y}) \SN{x - y}, \quad \text{for all } x,y \in \Rd, 
        \end{align*}
        \item\label[assumption]{ass:wp_growth} 
        satisfies for some constant $\cu >0$ the condition
        \begin{equation*}
            \CE(x) - \minobj \le \cu \,(1 + \SN{x}^2), \quad \text{for all } x \in \Rd,
        \end{equation*}
        and is either upper bounded by $\maxobj \coloneqq \sup_{x \in \Rd}\CE(x) < \infty$, or satisfies for some constants $\cl > 0$ and $\Cll\geq0$ the assumption
         \begin{equation*} 
		  \CE(x) - \minobj \geq \cl \SN{x}^2, \quad \text{for all } \SN{x} \ge \Cll.
	       \end{equation*}
    \end{enumerate}
\end{assumption}
\begin{assumption}
    \label{ass:ICP} 
    The objective function $\CE \in \CC(\Rd)$ satisfies, for some constants $\CE_{\infty}, R_0, \eta > 0$, and $\nu \in (0,\infty)$,
    \begin{itemize}
    \item in the case of isotropic noise ($\stoc = \stocis$), the conditions
    \begin{align*}
        \SN{x - \globmin} 
        &\le \frac{1}{\eta}\, \bigl( \CE(x) - \underbar{\CE} \bigr)^{\nu}, \quad
         \text{for all } x \in \B_{R_0}(\globmin), \\
    \CE_{\infty} 
        &< \CE(x) - \underbar{\CE}, \quad
         \text{for all } x \in \bigl(\B_{R_0}(\globmin)\bigr)^c;
    \end{align*}    
   \item in the case of anisotropic noise ($\stoc = \stocan$), the conditions
     \begin{align*}
        \N{x - \globmin}_{\infty} 
        &\le \frac{1}{\eta}\, \bigl( \CE(x) - \underbar{\CE} \bigr)^{\nu}, \quad
         \text{for all } x \in \B^{\infty}_{R_0}(\globmin), \\
    \CE_{\infty} 
        &< \CE(x) - \underbar{\CE}, \quad
         \text{for all } x \in \bigl(\B^{\infty}_{R_0}(\globmin)\bigr)^c. 
    \end{align*} 
    \end{itemize}
\end{assumption}

\begin{remark}[Discussion of assumptions on the objective $\CE$]
    The first part of \Cref{ass:wp}, Assumption~\ref{ass:wp_lowerbound}, states that the continuous objective function $\CE$ attains its infimum $\underbar{\CE}$ at some point $\globmin\in \bbR^d$, which is assumed to be unique.
    This is quantified (and, in fact, implied) by \Cref{ass:ICP}, which can be regarded as a tractability condition on the function landscape of the objective $\CE$.
    In the proximity of the global minimizer~$\globmin$, $\CE$ is assumed to be locally coercive.
    More precisely, in the cases of isotropic ($\stoc = \stocis$) or anisotropic ($\stoc = \stocan$) noise, respectively,
    $\CE$ is assumed to grow like $\SN{x - \globmin}^{1/\nu}$ or $\N{x - \globmin}_{\infty}^{1/\nu}$ in an Euclidean or $\ell_\infty$-ball of radius $R_0$ around $\globmin$.
    This assumption is also known as the inverse continuity condition, quadratic growth condition, or Hölderian error bound condition, see \cite{necoara2019linear} for more details.
    Outside these balls, i.e., in the farfield, $\CE$ is assumed to leave at least an $\CE_\infty$-wide gap to the minimal objective function value $\underbar{\CE}$.
    Together these two conditions imply uniqueness of $\globmin$ and quantify how challenging it is to locate $\globmin$ provided merely zero-order information about the objective function~$\CE$.

    The remaining parts of \Cref{ass:wp}, Assumptions~\ref{ass:wp_lipschitz}--\ref{ass:wp_growth}, are consistent with the ones made throughout the literature \cite{carrillo2018analytical,carrillo2019consensus,fornasier2024consensus,fornasier2021anisotropic,riedl2024perspective}.
    They require that $\CE$ is locally Lipschitz-continuous with Lipschitz constant allowed to have linear growth,
    and that $\CE$ is either bounded or at least quadratically growing.
    It is straightforward to extend the analysis to the slightly more general assumptions of \cite{gerber2023mean}.
\end{remark}

\subsection{Proof of the main result \Cref{thm:gap_tEM_globmin}}
\label{sec:key_auxiliary}

In preparation for the proof of our main theoretical result, \Cref{thm:gap_tEM_globmin},
we observe that the time-discrete CBO system \eqref{eq:EM-CBO} can be seen as an Euler–Maruyama time discretization \cite{higham2002strong,graham2013stochastic} of the system of SDEs
\begin{equation}
	\label{eq:tCBO}
		\d X^i_t = -\lambda \left( X^i_t - \conspoint{\empmeasure{t}} \right)
      \d t +  \sigma \stoc \left(X^i_t - \conspoint{\empmeasure{t}} \right) \d W^i_t, 
\end{equation}
where $\{W^{i}_t\}^{i \in \setnumber{N}}$ are independent $d$-dimensional Brownian processes and where $\empmeasure{t}$ denotes the empirical measure associated to the particles $\{ X^i_t \}^{i \in \setnumber{N}}$. Conversely but equivalently, \eqref{eq:tCBO} can be regarded as the continuous-time approximation of \eqref{eq:EM-CBO}.

Our proof strategy for  \Cref{thm:gap_tEM_globmin} is as follows. 
As has been noted already in several previous works on CBO, the lack of global Lipschitz continuity of the drift and diffusion coefficients of \eqref{eq:EM-CBO} poses several challenges in the derivation of quantitative estimates involving its  solution or the solution to its continuous-time approximation \eqref{eq:tCBO} \cite{carrillo2018analytical,huang2021MFLCBO,fornasier2020consensus_hypersurface_wellposedness,fornasier2024consensus,gerber2023mean,gerber2025uniform,kalise2022consensus,huang2024uniform}.
This absence also leads to difficulties regarding our goal of proving strong global mean square convergence for the Euler–Maruyama discretization \eqref{eq:EM-CBO}. Indeed, traditional finite-time convergence theory for numerical methods applied to SDEs necessitates a global Lipschitz condition on both the drift and diffusion coefficients \cite{higham2002strong,graham2013stochastic,platen1999introduction}.
Then, our approach is to combine techniques from the aforementioned classical theory with the strategy of discarding a low-probability event employed in \cite{gerber2023mean} to derive a quantitative mean-field result. More specifically, to prove \Cref{thm:gap_tEM_globmin}, we
\begin{itemize}
    \item proceed as in \cite[Theorem~3.8]{fornasier2024consensus} in the case of isotropic noise, and \cite[Theorem~3.19]{riedl2024perspective} in the case of anisotropic noise,
    and split up the error $\SN{\Xihat{K \dt} - \globmin}$ as the sum of a term estimating the gap between $\Xihat{K \dt}$ and the solution to a suitable mean-field dynamics (later introduced in \eqref{eq:tmfCBOcopycompact_integralform}), and a summand quantifying the distance between such mean-field dynamics and $\globmin$.
    \item Then, we control the second summand by a global convergence result in mean-field law, while addressing the first term through the previously mentioned approach that combines the classical SDE framework with the workaround of \cite{gerber2023mean}.
\end{itemize}
For clarity of exposition, we summarize the estimate of the gap between $\Xihat{K \dt}$ and the solution to the mean-field dynamics \eqref{eq:tmfCBOcopycompact_integralform} in \Cref{thm:gap_tEM_tmfCBO}. 
In the remainder of this section, we introduce the processes required for formulating its statement.
We postpone a more detailed discussion of the aforementioned technique, along with the proof of the theorem, to \Cref{sec:auxiliary}.

First, for compactness of notation, we define
\begin{equation}
\label{not: drift_diff}
\drift(x, \indivmeasure ) \coloneqq - \lambda (x - \conspoint{\indivmeasure })
\quad\text{ and }\quad
\diff(x, \indivmeasure ) \coloneqq \sigma\stoc(x - \conspoint{\indivmeasure }),
\end{equation}
for $x \in \Rd$, $\indivmeasure \in \CP(\Rd)$.
Considering \eqref{eq:tCBO} in the limit $N\to \infty$ leads to a mean-field formulation described by the McKean–Vlasov SDE
\begin{equation*}
    % \label{eq:tmfCBO}
    \d \Xbar{t}=\drift\left(\Xbar{t},\rho_t\right)\d t+ \diff \left(\Xbar{t},\rho_t\right)\d W_t,
\end{equation*}
where $\rho_t=\Law(\Xbar{t})$.
Consistently with what has been done in \cite{gerber2023mean}, we introduce $\{ \Xibar{t} \}^{i \in \setnumber{N}}_{t\in[0,T]}$, $N$ independent copies of the above SDE driven by the same Brownian motion of \eqref{eq:tCBO} and initialized by the same processes of \eqref{eq:tCBO} (or, equivalently, \eqref{eq:EM-CBO}). Their integral formulation is described by 
\begin{equation}
\label{eq:tmfCBOcopycompact_integralform}
		\Xibar{t} = X^i_0 + \int_0^t \, \drift\left(\Xibar{s}, \rho_{s}\right) \d s + \int_0^t\diff\left(\Xibar{s}, \rho_{s}\right) \d W^i_s,  %\quad \text{for } \; i \in \setnumber{N} \text{ and } t\geq 0.
\end{equation}
for any $i \in \setnumber{N}, t \in [0,T]$.

As outlined in the proof strategy, in \Cref{thm:gap_tEM_tmfCBO} we aim to estimate the discrepancy between the solution to the implementable Euler–Maruyama scheme \eqref{eq:EM-CBO} and the solution to the mean-field equation \eqref{eq:tmfCBOcopycompact_integralform}.
To accomplish this, we could directly consider the formulation provided by \eqref{eq:EM-CBO}, and combine a discrete-time induction argument with a discrete Gr\"onwall inequality in order to effectively control the approximation error. However, 
for conciseness, we follow the approach described in \cite{graham2013stochastic} and introduce a continuous-time extension of  \eqref{eq:EM-CBO}, which reads
\begin{equation}
	\label{eq:t-EM-CBO}
		\Xitilde{t} = \Xitilde{\disp{t}}+ \drift\left(\Xitilde{\disp{t}}, \tilempmeasure{\disp{t}}\right)(t-\disp{t})+ \diff\left(\Xitilde{\disp{t}}, \tilempmeasure{\disp{t}}\right)(W^i_{t}-W^i_{\disp{t}}),
\end{equation}
for $\; i \in \setnumber{N}$, $t\in[k \dt, (k+1)\dt)$ and 
\begin{equation}
\label{def:disp}
    \disp{t}=k \dt, \quad t\in[k\dt, (k+1)\dt) .
\end{equation}
In \eqref{eq:t-EM-CBO}, the notation $\tilempmeasure{t}$ denotes the empirical measure associated to $\{ \Xitilde{t} \}^{i \in \setnumber{N}}_{t\in[0,T]}$.
    The continuous-time extension \eqref{eq:t-EM-CBO} is constructed so that it coincides with the discrete solution \eqref{eq:EM-CBO} at the grid points $\{ k \dt\}_{k \in \setnumber{K}}$. More precisely, for every $k \in \setnumber{K}$ we have
    \begin{equation*}
        \Xitilde{\disp{t}}=\Xitilde{k \dt}= \Xihat{k \dt}.
    \end{equation*}
It is possible to write \eqref{eq:t-EM-CBO} in the integral form
\begin{equation}
\label{eq:t-EM-CBOintform}
    \Xitilde{t}=X_{0}^i+\int_0^t \drift\left(\Xitilde{\disp{s}}, \tilempmeasure{\disp{s}}\right)\, \d s + \int_0^t \diff\left(\Xitilde{\disp{s}}, \tilempmeasure{\disp{s}}\right)\, \d \W{s}. 
\end{equation}

We are now ready to state the result proving an estimate for the gap between $\Xitilde{t}$, solution to \eqref{eq:t-EM-CBO}, and $\Xibar{t}$, solution to \eqref{eq:tmfCBOcopycompact_integralform}.

\begin{theorem}
    \label{thm:gap_tEM_tmfCBO}
    Suppose that $\CE$ satisfies \Cref{ass:wp}. Let $\rho_0 \in \CP_q(\Rd)$ for some $q\ge 4$.
    Let $\{ \Xitilde{t} \}^i_t$ denote the solution to the dynamics defined in \eqref{eq:t-EM-CBO} with $\dt \le 1$ and $\{ \Xibar{t} \}^i_t$ to  \eqref{eq:tmfCBOcopycompact_integralform}.
    Then, there exist constants $\cna, \cmfa > 0$ (independent of $N, \dt$) such that, for any $i \in \setnumber{N}$,
    it holds
    \begin{equation*}
        \bbE \left( \sup_{t \in [0,T]} \SN{\Xitilde{t} - \Xibar{t}}^2 \right) \le \cna \dt + \cmfa N^{-\min\{1,\frac{q-2}{4}\}} \color{black} .
    \end{equation*}
    It holds $\cna=\cna(\Ce, \alpha, \lambda, \sigma, \kappa(\stoc), T, \rho_0, \cbt)$ and $\cmfa=\cmfa(\CE, \alpha, \lambda, \sigma, \kappa(\stoc), q, T, \rho_0) $, for $\kappa(\stoc)$ defined in \eqref{def:kappa}, 
    $T$ chosen according to Equation \eqref{eq:defT} and $\cbt$ the constant of \Cref{lemma: moment_est_brown}.
\end{theorem}

\begin{remark}
The constant $\cbt$, which $\cna$ depends on, grows linearly in the dimension $d$ by construction.
\end{remark}
\begin{remark}
    The result obtained for the Euler–Maruyama discretization can be generalized to the Milstein discretization, see e.g. \cite[Chapter 7]{graham1997stochastic}. In this case, the time discretization step enters the upper estimate in the form $(\Delta t)^2$, reflecting the well-known higher order of strong convergence (of the standard SDE framework) associated with the Milstein compared to the Euler–Maruyama scheme. 
\end{remark}
\begin{remark}
\Cref{thm:gap_tEM_tmfCBO} proves the strong mean square convergence of $\Xitilde{t}$ to $\Xibar{t}$. The proof technique can be further generalized to obtain strong $p$th convergence for any $p\in \left(0, \frac{q}{2}\right]$. 
Both constants $\cna$ and $\cmfa$ would then depend on $p$ and the statement would translate to
 \begin{equation*}
        \bbE \left( \sup_{t \in [0,T]} \SN{\Xitilde{t} - \Xibar{t}}^p \right) \le \cna (\dt)^{\frac{p}{2}} + \cmfa N^{-\min\{\frac{p}{2},\frac{q-p}{2p}\}}.
    \end{equation*}
\end{remark}

We are now ready to provide the proof of our main result, \Cref{thm:gap_tEM_globmin}.

\begin{proof}[Proof of \Cref{thm:gap_tEM_globmin}]
    Fix $0 < \dt \le 1$, $K \in \bbN^+$ and $\vartheta \in (0,1)$.
    Let us abbreviate $\xi := (1-\vartheta)/(1+\vartheta/2)$
    and define the time horizon
    \begin{equation*}
        T^*
        = \frac{1}{(1-\vartheta)\big(2\lambda - \kappa(D) \sigma^2\big)} \log\left(\left(\frac{\CV(\rho_0)}{\widetilde\varepsilon}\right)^{1/\xi}\right),
    \end{equation*}
    where $\widetilde\varepsilon$ is chosen such that $K\Delta t = \xi T^*$ holds.
    This can be achieved with the choice 
    \begin{equation*}
        \label{eq:proof:mainthm:4}
        \widetilde\varepsilon = \CV(\rho_0)\exp\left(-(1-\vartheta)\big(2\lambda - \kappa(D) \sigma^2\big)K\Delta t\right).
    \end{equation*}
    
    We can now apply 
    \cite[Theorem~3.7]{fornasier2024consensus} in the case of isotropic noise, and
    \cite[Theorem~3.6]{riedl2024perspective} in the case of anisotropic noise
    with accuracy
    \begin{equation}
        \label{eq:proof:mainthm:5}
        \varepsilon = \CV(\rho_0) (\widetilde\varepsilon/\CV(\rho_0))^{1/\xi},
    \end{equation}
    which satisfies $\varepsilon\in(0,\CV(\rho_0))$, and with time horizon $T^*$ such that $T^*=\frac{1}{(1-\vartheta)(2\lambda - \kappa(D) \sigma^2)} \log\left(\CV(\rho_0)/\varepsilon\right)$ by construction.
    The above-cited results show that there exists some $\alpha_0 > 0$ such that for $\alpha\geq\alpha_0$, there is $T\in[\xi T^*,T^*]$ with $\CV(\rho_T)=\varepsilon$, and that for all $t\in[0,T]$ it holds
    \begin{equation*}
        \CV(\rho_t)
        \leq \CV(\rho_0)\exp\left(-(1-\vartheta)\big(2\lambda - \kappa(D) \sigma^2\big)t\right),
    \end{equation*}
    where $\CV(\rho_t) \coloneqq \frac{1}{2} \int_{\Rd} \SN{x-\globmin}^2 \d \rho_t(x)$.
    Thus, in particular, since by the choice of $\widetilde\varepsilon$ it holds $K \Delta t = \xi T^*$, we have
    \begin{equation}
        \label{eq:proof:mainthm:9}
    \begin{split}
        \CV(\rho_{K \Delta t}) \leq \CV(\rho_0)\exp\left(-(1-\vartheta)\big(2\lambda - \kappa(D) \sigma^2\big)K\Delta t\right)
       =\widetilde\varepsilon.
    \end{split}
    \end{equation}

    Recalling that the solutions to \eqref{eq:EM-CBO} and \eqref{eq:t-EM-CBO} coincide at the discrete time points~$k\Delta t$,
    we obtain with triangle inequality for any $i\in\setnumber{N}$ that
    \begin{align}
        \label{eq:proof:mainthm:10}
         \bbE \MagPar{\SN{\Xihat{K \dt} - \globmin}^2}
        &=
        \bbE \MagPar{\SN{\Xitilde{K\Delta t} - \globmin}^2} \nonumber\\
        &\le 2 \bbE \MagPar{\SN{\Xitilde{{K\Delta t}} - \Xibar{{K\Delta t}}}^2} + 2 \bbE \MagPar{\SN{\Xibar{{K\Delta t}} - \globmin}^2} \nonumber\\
        &= 2 \bbE \MagPar{\SN{\Xitilde{{K\Delta t}} - \Xibar{{K\Delta t}}}^2} + 4 \CV(\rho_{K \Delta t}),
    \end{align}
    where in the last equality we have used the definition of $\CV(\rho_{K \Delta t})$.
    In order to estimate the first summand on the right-hand side of \eqref{eq:proof:mainthm:10}, we employ \Cref{thm:gap_tEM_tmfCBO}.
    The second summand can be bounded using \eqref{eq:proof:mainthm:9}, which concludes the proof.
\end{proof}

\begin{remark}
The proof of \Cref{thm:gap_tEM_globmin} leverages the bound on the distance between the mean-field solution $\Xibar{{K\Delta t}}$ of \eqref{eq:tmfCBOcopycompact_integralform} and the global minimizer $\globmin$ from \cite{fornasier2024consensus,riedl2024perspective},
incurring the assumption $\globmin \in \supp(\rho_0)$ for an initial datum with regularity $\rho_0\in\CP_q(\bbR^d)$.
In view of recent advances in consensus‑based optimization \cite{fornasier2025regularity,huang2025faithful}, the assumption $\globmin \in \supp(\rho_0)$, however, is no restriction, as it can be removed for initializations $\rho_0$ that have a sufficiently smooth density.
By adapting the arguments of \cite{fornasier2024consensus,riedl2024perspective} to these newer frameworks, our analysis can be extended straightforwardly to these slightly more general settings.
\end{remark}

\begin{remark}
\color{black}In practical applications of CBO, the time-discretization error (first summand on the right-hand side of~\eqref{eq:gap_tEM_globmin}) is typically not the dominant error contribution. Instead, since the number of particles $N$ is limited by the available computational budget, in particular the finite-particle-approximation error (second summand on the right-hand side of~\eqref{eq:gap_tEM_globmin}) typically dominates the $\CO(\dt)$ error term.
Nevertheless, the convergence rate with respect to $\dt$ predicted by \Cref{thm:gap_tEM_globmin} becomes visible once the finite-particle-approximation and optimization error contributions are sufficiently small. Thus, the $\CO(\dt)$ error contribution is critical only in regimes where the other sources of error are negligible. 
\end{remark}

%%%%%%%%%%%%%%%%%%%%%%%%%%%%%%%%%%%%%%%%%%%%%%%%%%
%%%%%%%%%% Section %%%%%%%%%%%%%%%%%%%%%%%%%%%%%%%
%%%%%%%%%%%%%%%%%%%%%%%%%%%%%%%%%%%%%%%%%%%%%%%%%%
% \newpage
\section{Proof of \Cref{thm:gap_tEM_tmfCBO}}
\label{sec:auxiliary}

The proof follows the lines of Sznitman's argument to derive quantitative mean-field results, see  \cite{sznitman1991propagation,chaintron2022propagation}, for the integral formulation \eqref{eq:t-EM-CBOintform} of the continuous-time extension process $\{\Xitilde{t}\}_{t}^{i}$. It addresses the absence of a global Lipschitz property for the coefficients of the dynamics of $\{\Xitilde{t}\}_{t}^{i}$ by removing an event of small probability, controllable through their moment estimates, as in \cite[Theorem~2.6]{gerber2023mean}. We estimate the moments of $\{\Xitilde{t}\}_{t}^{i}$ in \Cref{lemma:moment_estimate_tEMCBO}.
\Cref{coroll: prel,lemma:stability_drift_diff}  serve as preliminary results for \Cref{lemma:moment_estimate_tEMCBO} and are also relevant in other parts of \Cref{thm:gap_tEM_tmfCBO}.
Furthermore, to derive an explicit rate in $\dt$, we apply a method commonly used in the numerical analysis of the Euler–Maruyama scheme for SDEs with globally Lipschitz coefficients, see e.g. \cite[Chapter 7]{graham2013stochastic}. This approach requires estimating the gap $\bbE \SN{\Xitilde{t}-\Xitilde{\disp{t}}}$ of the extension process at times $t$ and $\disp{t}$. We adapt the classical theory to our setting in \Cref{lem: gap_t-EM-CBO_t-EM-CBOint}, with \Cref{lemma: moment_est_brown} serving as an auxiliary tool for its proof.

\begin{lemma}
 \label{coroll: prel}
     Suppose that $\CE$ satisfies \Cref{ass:wp_lowerbound,ass:wp_growth}.
     Let $\indivmeasure\in\mathcal{P}_q(\mathbb{R}^d)$ for some $q \geq 2$.
     Then, there exist constants $\cprel, \ccor >0$
     such that  
     \begin{align}
        \label{eq:conspoint_bound}
        \SN{\conspoint{\indivmeasure}}^q
        & \leq \cprel \int_{\Rd} \SN{x}^q \d\indivmeasure(x)
        \intertext{and}
        \label{eq:drift_diff_bound}
        \max\left\{\SN{\drift(x,\indivmeasure)}^q, \N{\diff(x,\indivmeasure)}_F^q\right\}
        & \leq \ccor \left(\SN{x}^q+\int_{\Rd} \SN{y}^q \d\indivmeasure(y)\right) \quad \text{for all } x\in\mathbb{R}^d.
     \end{align}
     It holds $\cprel=\cprel(\CE,\alpha,q)$ and $\ccor=\ccor(\lambda,\sigma,\kappa(\stoc),q,\cprel)$.
\end{lemma}

\begin{proof}
    Inequality~\eqref{eq:conspoint_bound} directly follows by applying \Cref{lemma: mom_bounds} with $r=1$.
    To prove Inequality~\eqref{eq:drift_diff_bound}, first note that
    \begin{align}
        \label{eq:coroll: prel:1}
        \SN{x - \conspoint{\indivmeasure}}^q
        \leq 2^{q-1}\left(\SN{x}^q+\SN{\conspoint{\indivmeasure}}^q\right)
        \leq 2^{q-1} 
        \left( \SN{x}^q + \cprel \int_{\Rd} \SN{y}^q \d\indivmeasure(y) \right),
    \end{align} 
    where we used \eqref{eq:conspoint_bound} to obtain the inequality in the last step.
    
    For the drift term,
    we have
    \begin{align*}
        \SN{\drift(x,\indivmeasure)}^q
        & = \SN{-\lambda (x-\conspoint{\indivmeasure})}^q
        = \lambda^q \SN{x-\conspoint{\indivmeasure}}^q,
    \end{align*}
    while for the diffusion term,
    we have 
    \begin{equation*}
        \N{\diff(x,\indivmeasure)}_F^q = \sigma^q \kappa(\stoc)^{\frac{q}{2}} \SN{x - \conspoint{\indivmeasure}}^q,
    \end{equation*}
    with $\kappa(\stoc)$ defined in \eqref{def:kappa}.
    The statement now follows by using \eqref{eq:coroll: prel:1}.
\end{proof}

\begin{lemma}
    \label{lemma:stability_drift_diff}
    Suppose that $\CE$ satisfies \Cref{ass:wp_lowerbound,ass:wp_growth}.
    Fix $q \geq 2$.
    Let $\{\Xitilde{t}\}_{t}^{i}$ 
    denote the solution to the continuous-time extension process~\eqref{eq:t-EM-CBO}.
    Then, for any $t\in[0,T]$ and for any $i \in \setnumber{N}$,
    it holds
    \begin{equation*}
        \max\left\{\bbE \MagPar{\SN{\drift\left(\Xitilde{t}, \tilempmeasure{t}\right)}^q}, \bbE \MagPar{\N{\diff\left(\Xitilde{t}, \tilempmeasure{t}\right)}_F^q}\right\} \leq \ccor \, \bbE \MagPar{\SN{\Xitilde{t}}^q},
    \end{equation*}
    where the constant $\ccor$ is as in \Cref{coroll: prel}.
\end{lemma}

\begin{proof}
    Leveraging \Cref{coroll: prel},
    we have
    \begin{equation*}
        \max\left\{\SN{\drift\left(\Xitilde{t}, \tilempmeasure{t}\right)}^q, \N{\diff\left(\Xitilde{t}, \tilempmeasure{t}\right)}_F^q\right\}
        \leq \ccor \left( \SN{\Xitilde{t}}^q + \int_{\Rd} \SN{y}^q \d \tilempmeasure{t} (y)\right).
    \end{equation*}
    By taking expectations on both sides,
    we obtain the statement since
    \begin{equation}
    \label{eq:nosup}
        \bbE \left(\int_{\Rd} \SN{y}^q \d \tilempmeasure{t} ( y)\right)
        = \bbE \MagPar{\SN{\Xitilde{t}}^q},
    \end{equation}
    which is due to the fact the common law of $(\Xtilde{1},\ldots,\Xtilde{N})$ is invariant under permutations of $\setnumber{N}$, for which reason its marginal laws (which correspond to the particles) are all the same. 
\end{proof}

\begin{lemma}
\label{lemma:moment_estimate_tEMCBO}
     Suppose that $\CE$ satisfies \Cref{ass:wp}.
     Let $\rho_0\in\CP_q(\Rd)$ for some $q\geq2$.
     Let $\{\Xitilde{t}\}_{t}^{i}$
     denote the solution to the continuous-time extension process~\eqref{eq:t-EM-CBO}.
     Then, there exists a constant $\cmom>0$ (independent of $\dt$, $N$) such that for any $i \in \setnumber{N}$
     \begin{equation}
     \label{eq:moment_estimate_tEMCBO}
         \max\left\{\bbE \left( \sup_{t\in [0,T]} \SN{\Xitilde{t}}^q\right), \bbE \left( \sup_{t\in [0,T]} \SN{\conspoint{\tilempmeasure{t}}}^q\right) \right\}
         \leq \cmom.
     \end{equation}
     It holds $\cmom=\cmom(\CE, \alpha, \lambda, \sigma, \kappa(\stoc),q,T,\rho_0)$.
\end{lemma}

\begin{proof}
    
    The proof follows the lines of \cite[Lemma~3.5]{gerber2023mean} but for the integral formulation \eqref{eq:t-EM-CBOintform} of the continuous-time extension process $\{\Xitilde{t}\}_{t}^{i}$
    of the discrete-time scheme~\eqref{eq:EM-CBO}.
    
    Let us fix $q \ge 2$ and $i \in \setnumber{N}$.
    Using the integral formulation \eqref{eq:t-EM-CBOintform} of the continuous-time extension process $\{\Xitilde{t}\}_{t}^{i}$ of \eqref{eq:t-EM-CBO},
    we can bound 
    \begin{equation*}
        \absbig{\Xitilde{t}}^q
        \lesssim \abs{X_{0}^i}^q+\abs{\int_0^t \drift\left(\Xitilde{\disp{s}}, \tilempmeasure{\disp{s}}\right) \d s}^q + \abs{\int_0^t \diff\left(\Xitilde{\disp{s}}, \tilempmeasure{\disp{s}}\right) \d \W{s}}^q
    \end{equation*}
    for any $t \in [0,T]$.
    By taking first the supremum over $t \in [0, T]$ and consecutively applying the expectation to both sides of the previous inequality, we obtain
    \begin{align*}
        &\bbE \left(\sup_{t\in[0,T]} \SN{\Xitilde{t}}^q\right)\\
        &\lesssim \bbE \MagPar{\SN{X_{0}^i}^q} + \bbE\left(\sup_{t\in[0,T]}\SN{\int_0^t \drift \left(\Xitilde{\disp{s}}, \tilempmeasure{\disp{s}}\right) \d s}^q\right)  \bbE\left(\sup_{t\in[0,T]}\SN{\int_0^t  \diff \left(\Xitilde{\disp{s}}, \tilempmeasure{\disp{s}} \right) \d \W{s}}^q\right) \\
        &\lesssim \bbE \MagPar{\SN{X_{0}^i}^q} + \bbE \MagPar{\left(\int_0^T \SN{\drift \left(\Xitilde{\disp{t}}, \tilempmeasure{\disp{t}}\right)} \d t\right)^q} \\
        & \qquad \qquad + \cbgd \, \bbE \MagPar{\left(\int_0^T \left\| \diff \left(\Xitilde{\disp{t}}, \tilempmeasure{\disp{t}} \right)\right\|^2_F \d t\right)^{q/2}},
    \end{align*}
    where the second step uses the Burkholder–Davis–Gundy inequality, see~\cite[Theorem~7.3]{mao2007stochastic}, for a constant $\cbgd$ depending only on $q$.
    An application of Jensen's inequality and \Cref{lemma:stability_drift_diff} yields
    \begin{align*}
        &\bbE \left(\sup_{t\in[0,T]} \SN{\Xitilde{t}}^q\right)\\
        &\quad\lesssim\bbE \MagPar{\SN{X_{0}^i}^q}+T^{q-1}\int_0^T \bbE \MagPar{\SN{\drift \left(\Xitilde{\disp{t}}, \tilempmeasure{\disp{t}}\right)}^q} \d t\\
        & \qquad \qquad + \cbgd T^{q/2-1}\int_0^T \bbE \MagPar{  \left\| \diff(\Xitilde{\disp{t}}, \tilempmeasure{\disp{t}})\right\|_F^q}\d t\\
        &\quad\lesssim\bbE \MagPar{\SN{X_{0}^i}^q} + C \int_0^T \bbE \MagPar{\SN{\Xitilde{\disp{t}}}^q} \d t
    \end{align*}
    for a constant $C$ depending on $q,T$ and $\ccor$ of \Cref{lemma:stability_drift_diff}.
    Since $\disp{t}\leq t$ as of \eqref{def:disp},
    it holds
    \begin{equation*}
        \SN{\Xitilde{\disp{t}}} \leq \sup_{s\in [0,t]} \SN{\Xitilde{s}},
    \end{equation*}
    which allows to estimate
    \begin{align*}
        \bbE \left(\sup_{t\in[0,T]} \SN{\Xitilde{t}}^q\right)
        &\lesssim \bbE \MagPar{\SN{X_{0}^i}^q} + C \int_0^T \bbE\left(\sup_{s\in [0,t]} \SN{\Xitilde{s}}^q \right) \d t.
    \end{align*}
    An application of Grönwall's inequality concludes the first part of the statement \eqref{eq:moment_estimate_tEMCBO},
    which introduces the constant $\cmom$ depending on $(\CE, \alpha, \lambda, \sigma, \kappa(\stoc), q, T, \rho_0)$.
    
    The second part follows immediately with Equation~\eqref{eq:conspoint_bound} from \Cref{coroll: prel}.
\end{proof}

\begin{lemma}
    \label{lemma: moment_est_brown}
    Let $\{W_{t}\}_{t \ge 0}$ denote a $d$-dimensional Brownian motion and fix $q\geq 2$. Then, there exists a constant $\cbm > 0$ such that for any $t,s \ge 0$ it holds
    \begin{equation*}
        \bbE \SN{W_t-W_s}^q\leq \cbm (t-s)^{q/2},
    \end{equation*}
    where $\cbm$ 
    is given by 
    \begin{equation*}
        \cbm :=
    \begin{cases}
            d^\frac{q}{2} (q-1)!! \quad \text{if $q$ even,}\\
            d^\frac{q}{2} (q!!)^\frac{q}{q+1} \quad \text{ if $q$ odd.}
        \end{cases}
    \end{equation*}
\end{lemma}
\begin{proof}
By definition of $\{W_{t}\}_{t \ge 0}$, $W_t - W_s$ is a random $d$-dimensional normal vector with zero mean and covariance matrix $(t-s)I_d$.
 The statement follows directly from Hölder’s inequality combined with a classical formula for the $q$-th moment of a $1$‑dimensional normal variable, which states that 
    $\bbE\SN{Z}^{q} = (q-1)!! \,\xi^{\frac{q}{2}}$, 
for a $1$-dimensional normal variable  $Z$ with mean zero and variance $\xi>0$, $q \ge 2$ even, and 
where $(q-1)!!=(q-1)\cdots5\cdot3\cdot1$, see for instance \cite[Proposition~2.6.1]{vershynin2018high}. 
\end{proof}

\begin{lemma}
    \label{lem: gap_t-EM-CBO_t-EM-CBOint}
    Suppose that $\CE$ satisfies \Cref{ass:wp}.
    Let $\rho_0\in\mathcal{P}_q(\mathbb{R}^d)$ for some $q\geq2$.
    Let $\{\Xitilde{t}\}^i_t$ denote the solution to the continuous-time extension process \eqref{eq:t-EM-CBO}.
    Then, there exists a constant $\cgap>0$ (independent of $N, \Delta t$) such that  for any $i \in \setnumber{N}$
    \begin{equation*}
        \sup_{t \in [0,T]} \bbE \MagPar{\SN{\Xitilde{t}-\Xitilde{\disp{t}}}^q }\leq  \cgap \left( (\dt)^q+\cbm (\dt)^\frac{q}{2}\right).
    \end{equation*}
    It holds $\cgap=\cgap(\CE, \alpha, \lambda, \sigma, \kappa(\stoc), q, T, \rho_0)$ and $\cbm$ is the constant from \Cref{lemma: moment_est_brown}.
\end{lemma}

\begin{proof}
    This proof follows the strategy of \cite[Theorem 7.10]{graham2013stochastic} adapted to our context.

    Let us fix $q \ge 2$ and $i \in \setnumber{N}$.
    Using the formulation \eqref{eq:t-EM-CBO} of the continuous-time extension process $\{\Xitilde{t}\}_{t}^{i}$,
    we can estimate
    \begin{align*}
       \bbE \MagPar{\SN{\Xitilde{t}-\Xitilde{\disp{t}}}^q }
        &\lesssim \bbE \MagPar{\SN{ \drift\left(\Xitilde{\disp{t}}, \tilempmeasure{\disp{t}}\right)(t-\disp{t})}^q} \\
        &\qquad \qquad+\bbE \MagPar{\SN{ \diff\left(\Xitilde{\disp{t}}, \tilempmeasure{\disp{t}}\right) (W^i_t-W^i_{\disp{t}})}^q}. 
    \end{align*}
    The first term on the right-hand side of the inequality can be bounded as
    \begin{align*}
        \bbE \MagPar{\SN{\drift\left(\Xitilde{\disp{t}}, \tilempmeasure{\disp{t}})(t-\disp{t}\right)}^q} &= \bbE \MagPar{\SN{ \drift\left(\Xitilde{\disp{t}}, \tilempmeasure{\disp{t}}\right)}^q} \SN{t-\disp{t}}^q\\
        & \leq \ccor \, \bbE \MagPar{\SN{ \Xitilde{\disp{t}} }^q} \, (\dt)^q,
    \end{align*}
    where the last step used \Cref{lemma:stability_drift_diff} together with the fact that $\SN{t-\disp{t}} \le \dt$ by definition of $\gamma_{\Delta t}$ in \eqref{def:disp}.
    For the second term,
    utilizing the independence of $\Xitilde{\disp{t}}$ from $W^i_t-W^i_{\disp{t}}$ in the second step,
    we can bound
    \begin{align*}
        \bbE \MagPar{\SN{ \diff\left(\Xitilde{\disp{t}}, \tilempmeasure{\disp{t}}\right) (W^i_t-W^i_{\disp{t}}) }^q }
        &\leq \bbE \MagPar{\left(\left\| \diff\left(\Xitilde{\disp{t}}, \tilempmeasure{\disp{t}}\right) \right\|_F \SN{ W^i_t-W^i_{\disp{t}}}\right)^q} \\
        &= \bbE \MagPar{\left\| \diff\left(\Xitilde{\disp{t}}, \tilempmeasure{\disp{t}}\right) \right\|^q_F} \, \bbE \MagPar{\SN{ W^i_t-W^i_{\disp{t}}}^q }\\
        &\le \ccor \bbE \SN{ \Xitilde{\disp{t}} }^q \cbm (\dt)^{q/2},
    \end{align*}
    where the last step uses \Cref{lemma:stability_drift_diff} together with the fact that, as of \Cref{lemma: moment_est_brown}, it holds \\$\bbE \MagPar{\SN{ W^i_t-W^i_{\disp{t}}}^q} \le \cbm \SN{ t-\disp{t}}^{q/2} \le \cbm (\dt)^{q/2}$, where the last inequality is again by definition of $\gamma_{\Delta t}$ in \eqref{def:disp}.
    By combining the two estimates obtained above,
    we conclude the initial bound as 
    \begin{align*}
       \bbE \MagPar{\SN{\Xitilde{t}-\Xitilde{\disp{t}}}^q }
        &\lesssim \ccor \, \bbE \MagPar{\SN{ \Xitilde{\disp{t}} }^q} \, \left( (\dt)^q + \cbm (\dt)^{q/2} \right) \\
        &\lesssim \ccor \, \bbE \left( \sup_{t\in [0,T]} \SN{\Xitilde{t}}^q\right) \, \left( (\dt)^q + \cbm (\dt)^{q/2} \right) \\
        &\lesssim \ccor \, \cmom \left( (\dt)^q +\cbm (\dt)^{q/2} \right),
    \end{align*}
    where in the last inequality we used \Cref{lemma:moment_estimate_tEMCBO}. Setting $\cgap \coloneqq \ccor \, \cmom$, the statement follows.
\end{proof}

We are now ready to provide a proof of \Cref{thm:gap_tEM_tmfCBO}.

\begin{proof}[\Cref{thm:gap_tEM_tmfCBO}]
    Let us fix $q \ge 4$ and $i \in \setnumber{N}$.

    Using the integral formulation \eqref{eq:t-EM-CBOintform} of the continuous-time extension process $\{\Xitilde{t}\}_{t}^{i}$ of \eqref{eq:t-EM-CBO}
    and the integral formulation \eqref{eq:tmfCBOcopycompact_integralform} of the independent copies $\Xibar{t}$  of the mean-field dynamics,
    we can bound
    \begin{align*}
        \SN{\Xitilde{t}-\Xibar{t}}^2 
        &\lesssim \SN{ \int_0^t \left(\drift \left(\Xitilde{\disp{s}}, \tilempmeasure{\disp{s}}\right) - \drift \left(\Xibar{s}, \rho_s \right) \right) \d s}^2 \\
        &\hspace{1cm}+ \SN{ \int_0^t \left(\diff \left(\Xitilde{\disp{s}}, \tilempmeasure{\disp{s}}\right) - \diff \left(\Xibar{s}, \rho_s \right) \right) \d W_s^i}^2.
    \end{align*}
    By taking first the supremum over $t \in [0, T]$ and consecutively applying the expectation to both sides of the previous inequality, we obtain 
    \begin{align*}
        \mathbb{E}\left( \sup_{t\in [0,T]} \SN{\Xitilde{t}-\Xibar{t}}^2 \right)
        &\lesssim  \bbE \left( \sup_{t\in [0,T]} \SN{\int_0^t \left( \drift \left(\Xitilde{\disp{s}}, \tilempmeasure{\disp{s}}\right) - \drift \left(\Xibar{s}, \rho_s \right) \right) \d s }^2  \right) \\
        &\hspace{1cm}+  \bbE \left( \sup_{t\in [0,T]} \SN{\int_0^t \left( \diff \left(\Xitilde{\disp{s}}, \tilempmeasure{\disp{s}}\right) - \diff \left(\Xibar{s}, \rho_s \right) \right)\d W^i_s }^2\right)  \\
        &\lesssim \bbE \left( \int_0^T \SN{\drift \left(\Xitilde{\disp{t}}, \tilempmeasure{\disp{t}}\right) - \drift \left(\Xibar{t}, \rho_t \right)} \d t \right)^2 \\
       &\hspace{1cm} + \cbgdt \, \bbE \MagPar{\int_0^T \left\|\diff \left(\Xitilde{\disp{t}}, \tilempmeasure{\disp{t}}\right) - \diff \left(\Xibar{t}, \rho_t \right) \right\|_F^2 \d t}\\
         &\lesssim \, T \, \bbE \MagPar{\int_0^T \SN{\drift \left(\Xitilde{\disp{t}}, \tilempmeasure{\disp{t}}\right) - \drift \left(\Xibar{t}, \rho_t \right) }^2 \d t} \\
       &\hspace{1cm} + \cbgdt \, \bbE \MagPar{\int_0^T \left\|\diff \left(\Xitilde{\disp{t}}, \tilempmeasure{\disp{t}}\right) - \diff \left(\Xibar{t}, \rho_t \right) \right\|_F^2 \d t},
    \end{align*}
    where the second step uses the Burkholder–Davis–Gundy inequality, see~\cite[Theorem~7.3]{mao2007stochastic}, for a constant $\cbgd$ depending only on $q$.
    The last inequality follows from an application of Jensen's inequality.
    Utilizing the definitions of $\kappa(\stoc)$ and $\drift$, $\diff$ provided in \eqref{def:kappa} and \eqref{not: drift_diff} respectively, we can rearrange the inequality above to yield 
    \begin{equation}
        \label{eq:step_proof}
    \begin{aligned}
        &\mathbb{E}\left( \sup_{t\in [0,T]} \SN{\Xitilde{t}-\Xibar{t}}^2 \right) \\
        &\hspace{1cm} \lesssim \left( T \lambda^2 + \cbgdt \sigma^2 \kappa(\stoc) \right) \int_0^T \left( \bbE \MagPar{\SN{ \Xitilde{\disp{t}} - \Xibar{t} }^2} + \bbE \MagPar{\SN{  \conspoint{\tilempmeasure{\disp{t}}} -\conspoint{\rho_t}}^2} \right) \d t. 
    \end{aligned}
    \end{equation}
    
    Let us first focus on the term $\bbE \MagPar{\SN{  \conspoint{\tilempmeasure{\disp{t}}} -\conspoint{\rho_t}}^2}$. 
    Denoting by $\monopmeasure{t}$ the empirical measure associated to $\{\Xibar{t} \}^i_t$, we add and subtract $\conspoint{\monopmeasure{t}}$ to estimate
    \begin{equation}
    \label{eq:proof_cons}
        \bbE \MagPar{\SN{  \conspoint{\tilempmeasure{\disp{t}}} -\conspoint{\rho_t}}^2} \lesssim
        \bbE \MagPar{\SN{  \conspoint{\tilempmeasure{\disp{t}}} -\conspoint{\monopmeasure{t}}}^2} + \bbE \MagPar{\SN{  \conspoint{\monopmeasure{t}} -\conspoint{\rho_t}}^2}.
    \end{equation}
    Since $\rho_t \in \CP_q(\Rd)$ (with $ q \ge 4$) due to the regularity of $\rho_0$ and \Cref{lemma: exis_uniq_mf_tCBO}, the second summand on the right-hand side of the inequality can be estimated directly with \Cref{lemma: conv_weig_mean_iid_samp} as
    \begin{equation}
    \label{eq:proof_cons1}
        \bbE \MagPar{\SN{  \conspoint{\monopmeasure{t}} -\conspoint{\rho_t}}^2} \le \cwm N^{-1},
    \end{equation}
    for some constant $\cwm = \cwm(\CE, \alpha, q)$. 
    For the estimation of the first summand, we follow the strategy adopted in \cite{gerber2023mean}. Therefore set
    \begin{equation*}
        Z^i=\sup_{t\in[0,T]} \SN{\Xibar{t}}^2,
    \end{equation*}
    and observe that, thanks to \Cref{lemma: exis_uniq_mf_tCBO}, $\cbmmf$ dependent on $(\CE, \alpha, \lambda, \sigma, \kappa(\stoc), q, T, \rho_0)$ is such that $\cbmmf+1 > \bbE (Z^i)$.
    Let us now define  the event
    \begin{equation*}
        \Omega_{N,t} \coloneqq \left\{ \omega\in\Omega : \frac{1}{N} \sum_{i=1}^N \SN{\Xibar{t}(\omega)}^2 \geq \cbmmf \right\}.
    \end{equation*}
    and distinguish between two cases.
    \begin{itemize}
        \item On the set $\Omega \setminus \Omega_{N,t}$, it holds $\monopmeasure{t}(\omega) \in \mathcal{P}_{2,\cbmmf+1} (\mathbb{R}^d)$, as 
        \begin{equation*}
       \int \SN{x}^2 \,\d\monopmeasure{t}(\omega)(x) = \frac{1}{N} \sum_{i=1}^N \SN{ \Xibar{t}(\omega) }^2 < \cbmmf+1.
        \end{equation*}
        We may thus apply \Cref{lemma: stab_estim_weig_mean} for $\CE$ satisfying \Cref{ass:wp} to conclude that for some constant $\cstab = \cstab(\CE,\alpha,\cbmmf)$ it holds
    \begin{align}
        \bbE \left( \SN{ \conspoint{\tilempmeasure{\disp{t}}} - \conspoint{\monopmeasure{t}}}^2 \mathbf{1}_{\Omega \setminus \Omega_{N,t}}\right) &\leq \cstab \bbE \MagPar{\CW^2_2(\tilempmeasure{\disp{t}},\monopmeasure{t})} \nonumber\\
        & \leq \cstab \, \bbE \left( \frac{1}{N} \sum_{i=1}^N \SN{\Xitilde{\disp{t}} - \Xibar{t} }^2 \right) \nonumber \\
        &= \cstab \, \bbE \MagPar{\SN{\Xitilde{\disp{t}} - \Xibar{t} }^2},  \label{eq:consistency_bound}
    \end{align}
    where in the second line we have used the optimality of the coupling in the definition of the $\CW_2$ distance, along with the observation from Equation~\eqref{eq:nosup} to obtain the last equality.
    \item On the set $\Omega_{N,t}$, on the other hand, the stability estimate of \Cref{lemma: stab_estim_weig_mean} does not apply.
    Nevertheless, we observe that H\"older's inequality yields 
    \begin{align}
        \bbE \left( \SN{ \conspoint{\tilempmeasure{\disp{t}}} - \conspoint{\monopmeasure{t}}}^2 \mathbf{1}_{\Omega_{N,t}}\right)  
        &\leq \left( \bbE \MagPar{\SN{ \conspoint{\tilempmeasure{\disp{t}}} - \conspoint{\monopmeasure{t}}}^q}\right)^\frac{2}{q} \left( \bbP (\Omega_{N,t})\right)^\frac{q-2}{q} \nonumber \\
        &\lesssim \left( \bbE \MagPar{\SN{ \conspoint{\tilempmeasure{\disp{t}}}}^q} + \bbE \MagPar{\SN{ \conspoint{\monopmeasure{t}}}^q }\right)^\frac{2}{q} \left( \bbP (\Omega_{N,t})\right)^\frac{q-2}{q} \label{eq:proof_item}
    \end{align}
    and it remains to estimate the three terms of \eqref{eq:proof_item} separately. 
    
    If $\rho_0\in \mathcal{P}_q(\mathbb{R}^d)$, with $q \geq 4$, it holds $\mathbb{E}\left(\SN{Z^i}^r\right)<+\infty$ for $r=\frac{q}{2} \ge 2$ thanks to \Cref{lemma: exis_uniq_mf_tCBO}. As a result, we may apply \Cref{lemma: prob_bound_larg_excur} to estimate
    \begin{equation}
    \label{eq:proof_item1}
    \mathbb{P}(\Omega_{N,t})\leq \mathbb{P} \left( \frac{1}{N} \sum_{i=1}^N Z^i\geq \cbmmf+1 \right) \leq \cle \, N^{-\frac{q}{4}},
    \end{equation}
    for some constant $\cle = \cle(q,\cbmmf)$.
    
    Moreover, thanks to \Cref{lemma:moment_estimate_tEMCBO}, it holds that
    \begin{equation}
    \label{eq:proof_item2}
        \bbE \MagPar{\SN{ \conspoint{\tilempmeasure{\disp{t}}}}^q} \le 
        \bbE \left( \sup_{t \in [0,T]} \SN{ \conspoint{\tilempmeasure{t}} } \right) \leq \cmom (\CE, \alpha, \lambda, \sigma, \kappa(\stoc), q, T, \rho_0).
    \end{equation}
    Furthermore, Equation~\eqref{eq:conspoint_bound} of \Cref{coroll: prel}, in conjunction with \Cref{lemma: exis_uniq_mf_tCBO}, leads to
    \begin{align}
        \bbE \MagPar{\SN{ \conspoint{\monopmeasure{t}}}^q} &\le \cprel \, \bbE \left( \int_{\mathbb{R}^d} \SN{x}^q \d\monopmeasure{t}(x)\right) = \cprel \, \frac{1}{N} \sum_{i=1}^N \bbE \MagPar{\SN{ \Xibar{t} }^q }\nonumber\\
        &\le \cprel \bbE \left( \sup_{t\in[0,T]} \SN{\Xibar{t}}^q \right) 
        %\le \cprel (\cbmmf+1),
        \le \cprel \cbmmf. \label{eq:proof_item3}
    \end{align}
    
    From this point forward we denote by $\cmom$ any constant that depends on $(\CE, \alpha, \lambda, \sigma, \kappa(\stoc), q, T, \rho_0)$ while being independent of $N$ and $\dt$. We  note that also the constants $\cle$ and $\cprel \cbmmf$ can be replaced by $\cmom$. Inserting \eqref{eq:proof_item1}-\eqref{eq:proof_item3} into \eqref{eq:proof_item} yields
    \begin{equation}
        \label{eq:moment_prob_bound}
        \bbE \left( \SN{ \conspoint{\tilempmeasure{\disp{t}}} - \conspoint{\monopmeasure{t}}}^2\mathbf{1}_{\Omega_{N,t}}\right)  
        \lesssim \cmom N^{-\frac{q-2}{4}}.
    \end{equation}
    \end{itemize}
    Putting \eqref{eq:consistency_bound} and \eqref{eq:moment_prob_bound} together, we get 
    \begin{align}
    \label{eq:proof_cons2}
        \bbE \MagPar{\SN{  \conspoint{\tilempmeasure{\disp{t}}} -\conspoint{\monopmeasure{t}}}^2 }
        &= \bbE \left( \SN{ \conspoint{\tilempmeasure{\disp{t}}} - \conspoint{\monopmeasure{t}}}^2 \mathbf{1}_{\Omega \setminus \Omega_{N,t}}\right)\nonumber\\
        &\quad\,+ \bbE \left( \SN{ \conspoint{\tilempmeasure{\disp{t}}} - \conspoint{\monopmeasure{t}}}^2\mathbf{1}_{\Omega_{N,t}}\right) \nonumber \\
        &\lesssim
        \cmom \left( \bbE \MagPar{\SN{\Xitilde{\disp{t}} - \Xibar{t} }^2} +  N^{-\frac{q-2}{4}} \right). 
    \end{align}
    Eventually, by combining \eqref{eq:proof_cons1} and \eqref{eq:proof_cons2}, we can continue \eqref{eq:proof_cons} to obtain
    \begin{align*}
        \bbE \MagPar{\SN{  \conspoint{\tilempmeasure{\disp{t}}} -\conspoint{\rho_t}}^2 }
        &\lesssim
        \bbE \MagPar{\SN{  \conspoint{\tilempmeasure{\disp{t}}} -\conspoint{\monopmeasure{t}}}^2} + \bbE \MagPar{\SN{  \conspoint{\monopmeasure{t}} -\conspoint{\rho_t}}^2}\\
        &\lesssim \cmom \left( \bbE \MagPar{\SN{\Xitilde{\disp{t}} - \Xibar{t} }^2} +  N^{-\frac{q-2}{4}} + N^{-1} \right),
    \end{align*}
    which provides the desired estimate for the consensus points.
    Plugging it into 
    \eqref{eq:step_proof} gives us
    \begin{equation}
    \label{eq:step_proof2}
    \begin{split}
        \mathbb{E}\left( \sup_{t\in [0,T]} \SN{\Xitilde{t}-\Xibar{t}}^2 \right) &\lesssim
        \cmom
        \left( T \lambda^2 + \cbgdt \sigma^2 \kappa(\stoc) \right) \int_0^T \bbE \MagPar{\SN{ \Xitilde{\disp{t}} - \Xibar{t} }^2} \d t \\
        &\hspace{1cm} + \cmom  N^{-\min \left\{\frac{q-2}{4},1 \right\}}.
    \end{split}
    \end{equation}
    In order to obtain an explicit rate in $\dt$, we observe that by triangle inequality
    \begin{equation*}
        \bbE \MagPar{\SN{\Xitilde{\disp{t}} - \Xibar{t}}^2} \lesssim \bbE \MagPar{\SN{\Xitilde{\disp{t}} - \Xitilde{t}}^2} + \bbE \MagPar{\SN{\Xitilde{t} - \Xibar{t}}^2}.
    \end{equation*}
    The first summand on the right-hand side can be estimated through \Cref{lem: gap_t-EM-CBO_t-EM-CBOint} for $q=2$ (since by assumption $\dt \leq 1$). This gives 
    \begin{equation*}
        \bbE \MagPar{\SN{\Xitilde{\disp{t}} - \Xibar{t} }^2}  \leq \cgap  \left( (\Delta t)^2 + \cbt \, \Delta t\right) + \bbE \MagPar{\SN{\Xitilde{t} - \Xibar{t}}^2}
         \leq \cgap  \left( 1+ \cbt \right) \Delta t + \bbE \MagPar{\SN{\Xitilde{t} - \Xibar{t}}^2},
    \end{equation*}
    where the second inequality uses $\dt\leq1$.
    Denoting by $\cna$ and $\cmfa$ any constants that are dependent on $(\CE, \alpha, \lambda, \sigma, \kappa(\stoc), T, \rho_0, \cbt)$ and $(\CE,\alpha,\lambda,\sigma, \kappa(\stoc),q,T,\rho_0)$, respectively, we plug the latter estimate into Equation~\eqref{eq:step_proof2} to estimate
    \begin{align*}
        \bbE\left( \sup_{t\in [0,T]} \SN{\Xitilde{t}-\Xibar{t}}^2 \right) 
        &\lesssim \cna \dt +
        \cmfa \int_0^T \bbE \MagPar{\SN{ \Xitilde{t} - \Xibar{t} }^2} \d t + 
        \cmfa  N^{-\min \left\{\frac{q-2}{4},1 \right\}} \\
        &\lesssim \cna \dt + \cmfa \int_0^T \bbE \left( \sup_{s\in[0,t]} \SN{\Xitilde{s} - \Xibar{s}}^2 \right) \d t + 
        \cmfa  N^{-\min \left\{\frac{q-2}{4},1 \right\}},
    \end{align*}
    where in the second line we have taken the supremum over $s \in [0,t]$.
    An application of Gr\"onwall's inequality concludes the proof.
\end{proof}

\begin{remark}
    In \Cref{thm:gap_tEM_tmfCBO} we establish a bound on the gap between $\Xitilde{t}$, solution to the continuous-time extension \eqref{eq:t-EM-CBO} of the implementable Euler–Maruyama scheme \eqref{eq:EM-CBO}, and $\Xibar{t}$, solution to the McKean–Vlasov SDE \eqref{eq:tmfCBOcopycompact_integralform}. Since \eqref{eq:t-EM-CBO} can be interpreted as {\color{black} an Euler–Maruyama discretization} of a McKean–Vlasov process, it is natural to ask how our  result \Cref{thm:gap_tEM_tmfCBO} relates to the existing literature on Euler–Maruyama approximations for McKean–Vlasov SDEs.
    
    The classical theory on Euler–Maruyama schemes for McKean–Vlasov SDEs assumes that the drift and diffusion coefficients are globally Lipschitz (or Hölder‑continuous) with respect to the state variable and globally Lipschitz with respect to the measure, see, for example, \cite{dos2022simulation, bao2022approximations}. More recent convergence results relax the regularity in the state variable, allowing local Lipschitz or monotonicity conditions, while retaining global Lipschitz continuity in the measure argument \cite{li2023strong, liu2023tamed}.
    Allowing only locally Lipschitz dependence on the law is substantially more delicate, as observed for instance in \cite{leobacher2022well}. 
    The McKean–Vlasov SDE \eqref{eq:tmfCBOcopycompact_integralform} that underlies our CBO algorithm belongs exactly to this more demanding class. Indeed, because of \Cref{lemma: stab_estim_weig_mean}, both the drift and the diffusion are merely locally Lipschitz with respect to the measure. 
    Our proof of \Cref{thm:gap_tEM_tmfCBO} resorts to a localization technique (w.r.t.\@ the probability space~$\Omega$ by introducing the event $\Omega_{N,t}$)  specifically tailored to cope with this limited regularity of the CBO dynamics, which, to the best of our knowledge, is not covered by existing results for McKean–Vlasov SDEs.
\end{remark}

%%%%%%%%%%%%%%%%%%%%%%%%%%%%%%%%%%%%%%%%%%%%%%%%%%
%%%%%%%%%% Section %%%%%%%%%%%%%%%%%%%%%%%%%%%%%%%
%%%%%%%%%%%%%%%%%%%%%%%%%%%%%%%%%%%%%%%%%%%%%%%%%%
%\newpage
\section{Conclusion} \label{sec:conclusion}

In this paper, we establish strong mean square convergence of the time-discrete Euler–Maruyama numerical scheme used to implement the CBO algorithm towards the global minimizer of a nonconvex nonsmooth objective function~$\CE$.
Demonstrating such convergence is challenging due to the lack of a global Lipschitz continuity condition of the drift and diffusion coefficients. 
We tackle this difficulty by following Sznitman's classical argument to derive quantitative mean-field results along with a recently introduced technique of discarding a low-probability moment-estimate-controlled event. 
We combine it with traditional finite-time convergence theory for numerical methods applied to SDEs.

Our main result showcases an explicit rate of convergence in the time discretization step $\dt$ and in the number of particles $N$. 
It improves over several existing results available in the literature, which either provide local convergence results under restrictive conditions on the initial agent configuration, lack the aforementioned explicit rate, or present statements that are either probabilistic or pertain to the continuous-time formulation of the CBO algorithm.

%%%%%%%%%%%%%%%%%%%%%%%%%%%%%%%%%%%%%%%%%%%%%%%%%%
%%%%%%%%%% Acknowledgements %%%%%%%%%%%%%%%%%%%%%%
%%%%%%%%%%%%%%%%%%%%%%%%%%%%%%%%%%%%%%%%%%%%%%%%%%
\subsection*{Acknowledgements}
SB and SV would like to thank Michael Herty and Lorenzo Pareschi for the valuable discussions during the final stages of the preparation of the manuscript. The authors would further like to thank Giacomo Borghi for helpful discussions during the manuscript's revision.

The work of SB is funded by the Deutsche Forschungsgemeinschaft (DFG, German Research Foundation) – 320021702/GRK2326 – Energy, Entropy, and Dissipative Dynamics (EDDy).
SB is a member of the INdAM Research National Group of Mathematical Physics (INdAM-GNFM).
The work of SV is supported by the European Union’s Horizon Europe research and innovation program under the Marie Sklodowska-Curie Doctoral Network Datahyking (Grant No. 101072546).
SV is a member of the INdAM Research National Group of Scientific Computing (INdAM-GNCS).

For the purpose of Open Access, the authors have applied a CC BY public copyright license to any Author Accepted Manuscript (AAM) version arising from this submission.

%%%%%%%%%%%%%%%%%%%%%%%%%%%%%%%%%%%%%%%%%%%%%%%%%%
%%%%%%%%%% Appendix %%%%%%%%%%%%%%%%%%%%%%%%%%%%%%%
%%%%%%%%%%%%%%%%%%%%%%%%%%%%%%%%%%%%%%%%%%%%%%%%%%
%\newpage
\appendix

\section{Technical auxiliary results}
\label{app:resultsGHV23}

For the reader's convenience,
we provide in this section several technical auxiliary results.

The first lemma is a generalization of \cite[Lemma~3.3]{carrillo2018analytical}.

\begin{lemma}[{\cite[Proposition A.4  (special case $u=\ell=2$, $C_1=0$)]{gerber2023mean}}: Bound on the weighted moment]
    \label{lemma: mom_bounds}
    % WB = weighted bound
    Suppose that $\CE$ satisfies \Cref{ass:wp_lowerbound,ass:wp_growth}.
    Let $\indivmeasure\in\mathcal{P}_q(\mathbb{R}^d)$ for some $q > 0$ and fix $0< r \le q$.
    Then, there exists a constant $\cprel=\cprel(\alpha, \cu, \cl, \Cll, r, q)$ such that
    \begin{equation*}
        \frac{\int_{\Rd} \SN{x}^r e^{-\alpha \CE(x)}\d \indivmeasure( x)}{\int_{\Rd} e^{-\alpha \CE(x)}\d \indivmeasure( x)}\leq \cprel \left( \int_{\Rd} \SN{x}^q \d \indivmeasure( x)\right)^{\frac{r}{q}}.
        \end{equation*}
\end{lemma}

The second lemma is a generalization of \cite[Theorem 3.2]{carrillo2018analytical}. In addition to the statement in \cite[Theorem 2.3]{gerber2023mean}, we provide the explicit dependence on the constants that influence the upper bound of the moments of the mean-field SDE, as this will be necessary for the proof of \Cref{thm:gap_tEM_tmfCBO}.

\begin{lemma}[{\cite[Theorem 2.3]{gerber2023mean}}: Existence, uniqueness and bound on the moments of the mean-field SDE]
\label{lemma: exis_uniq_mf_tCBO}   
Suppose that $\CE$ satisfies \Cref{ass:wp} and fix a final time $T>0$. Let $\rho_0\in\CP_q(\Rd)$ for some $q \ge 2$.  
Then, there exists a unique strong solution $\{ \Xbar{t} \}_{t \in [0,T]}$ with initial condition sampled from $\rho_0$.  
Furthermore, the process satisfies
\begin{equation*}
    \bbE \left( \sup_{t\in[0,T]} \SN{\Xbar{t}}^q \right) \le \cbmmf <  \infty, 
\end{equation*}
for some constant $\cbmmf = \cbmmf(\CE, \alpha, \lambda, \sigma, \kappa(\stoc), q, T, \rho_0)$, with $\kappa(\stoc)$ defined in \eqref{def:kappa}.
\end{lemma}

The remaining lemmas are used solely in the proof of \Cref{thm:gap_tEM_tmfCBO}.

\begin{lemma}[{\cite[Lemma 2.5]{gerber2023mean}}: Bound on the probability of large excursions]
\label{lemma: prob_bound_larg_excur} 
    Let $\{ Z^i \}^{i \in \setnumber{N}}$ be a family of real-valued i.i.d.\@ random variables such that $\bbE \MagPar{ \SN{Z^1}^r }<\infty$ for some $r \ge 2$. Fix $R>\bbE \SN{Z^1}$. Then, there exists a constant $\cle = \cle(r,R)>0$ such that
    \begin{equation*}
        \bbP\left( \frac{1}{N} \sum_{i=1}^N Z^i \geq R\right) \leq \cle N^{-\frac{r}{2}}.
    \end{equation*}
\end{lemma}
\begin{lemma}[{\cite[Corollary 3.3]{gerber2023mean}}: Stability estimate for the weighted mean]
\label{lemma: stab_estim_weig_mean}
     Suppose that $\CE$ satisfies \Cref{ass:wp}. Fix $R>0$. Then, there exists a constant $\cstab = \cstab(\CE,\alpha,R)$ such that for all $(\indivmeasure,\nu)\in \mathcal{P}_{2}(\mathbb{R}^d)\times \mathcal{P}_{2,R}(\mathbb{R}^d)$ it holds
     \begin{equation*}
         \SN{\conspoint{\indivmeasure}-\conspoint{\nu}} \leq \cstab \CW_2(\indivmeasure,\nu).
     \end{equation*}
\end{lemma}
\begin{lemma}[{\cite[Lemma 3.7]{gerber2023mean}}: Convergence of the weighted mean for i.i.d.\@ samples]
\label{lemma: conv_weig_mean_iid_samp} 
Suppose that $\CE: \Rd \to \mathbb{R}$ is bounded from below.
Let $\indivmeasure\in\mathcal{P}_q(\mathbb{R}^d)$ for some $q > 2$.
Then, there exists a constant $\cwm = \cwm(\CE,\alpha,q)$ such that 
\begin{equation*}
    \bbE \MagPar{ \SN{\conspoint{\overline{\indivmeasure}^N}-\conspoint{\indivmeasure}}^2 } \leq \cwm N^{-1}, 
\end{equation*}
where $\overline{\indivmeasure}^N\coloneqq \frac{1}{N}\sum_{n=1}^N \delta_{\overline{X}^n}$ and $\{ \overline{X}^n \}^{n\in\mathbb{N}}$ are i.i.d.\@ samples from $\indivmeasure$.
\end{lemma}

%%%%%%%%%%%%%%%%%%%%%%%%%%%%%%%%%%%%%%%%%%%%%%%%%%%%%%%%%%%%%
%%%%%%%%%%%%%%%%%%%%%%%%%%%%%%%%%%%%%%%%%%%%%%%%%%%%%%%%%%%%%
\section{Constants appearing in the manuscript}
\label{app:constants}
In this appendix, we group the constants that have been utilized throughout the manuscript. We specify their dependencies, as well as the results in which they have been used.
Please note that whenever a constant depends on $\CE$, it means that it depends on some of the constants that appear in \Cref{ass:wp,ass:ICP}. In addition, for the constants in the lemmas of \Cref{app:resultsGHV23}, we specify their dependencies as they are utilized in the manuscript.

\begin{table}[ht]
\centering
\renewcommand{\arraystretch}{1.3}
\small
\begin{tabular}{|l|l|l|}
\hline
\textbf{Constant} & \textbf{Related results} & \textbf{Dependencies} \\ \hline
$\kappa(\stoc)$ &  & $d$, if $\stoc = \stocis$\\ \hline
$\cna$ & \Cref{thm:gap_tEM_globmin,thm:gap_tEM_tmfCBO}& $\Ce, \alpha, \lambda, \sigma, \kappa(\stoc), T, \rho_0, \cbt$ \\ \hline
$\cmfa$ & \Cref{thm:gap_tEM_globmin,thm:gap_tEM_tmfCBO}& $\CE, \alpha, \lambda, \sigma, \kappa(\stoc), q, T, \rho_0$ \\ \hline
$\cprel$ & \Cref{coroll: prel,lemma: mom_bounds} & $\CE,\alpha,q$ \\ \hline
$\ccor$ & \Cref{coroll: prel,lemma:stability_drift_diff}& $\CE, \alpha, \lambda, \sigma, \kappa(\stoc), q$ \\ \hline
$\cmom$ & \Cref{lemma:moment_estimate_tEMCBO}& $\CE, \alpha, \lambda, \sigma, \kappa(\stoc),q,T,\rho_0$ \\ \hline
$\cbm$ & \Cref{lemma: moment_est_brown}& $d,q$ \\ \hline
$\cgap$ & \Cref{lem: gap_t-EM-CBO_t-EM-CBOint}& $\CE, \alpha, \lambda, \sigma, \kappa(\stoc), q, T, \rho_0$ \\ \hline 
$\cbmmf$ &\Cref{lemma: exis_uniq_mf_tCBO} &  $\CE, \alpha, \lambda, \sigma, \kappa(\stoc), q, T, \rho_0$\\ \hline
$\cle$ & \Cref{lemma: prob_bound_larg_excur}& $q, \cbmmf$ \\ \hline
$\cstab$ & \Cref{lemma: stab_estim_weig_mean}  & $\CE, \alpha, \cbmmf$ \\ \hline
$\cwm$ & \Cref{lemma: conv_weig_mean_iid_samp} & $\CE, \alpha$ \\ \hline
\end{tabular}
\end{table}

%%%%%%%%%%%%%%%%%%%%%%%%%%%%%%%%%%%%%%%%%%%%%%%%%%
%%%%%%%%%% Bibliography %%%%%%%%%%%%%%%%%%%%%%%%%%
%%%%%%%%%%%%%%%%%%%%%%%%%%%%%%%%%%%%%%%%%%%%%%%%%%
\bibliographystyle{abbrv}
\bibliography{bibliography.bib}

\end{document}